\documentclass[11pt]{ip-journal}
\usepackage{graphicx, amssymb, amsmath, amsthm}
\usepackage{epsfig}
\usepackage{hyperref}
\usepackage{mathrsfs}
\numberwithin{equation}{section}
\usepackage{amsthm}
\usepackage{subfigure}\usepackage{tikz}
\usepackage{here}
\usepackage{float}\usepackage{pifont}

\usepackage{enumitem}
\usetikzlibrary{matrix,arrows}
\usetikzlibrary{shapes}                            
\usetikzlibrary{calc}
\usetikzlibrary{arrows}
\usetikzlibrary{decorations.pathreplacing,decorations.markings}
\usetikzlibrary{patterns}

\def\p{\partial}

\newtheorem{theorem}{Theorem}[section]

\newtheorem{lemma}[theorem]{Lemma}

\newtheorem{corollary}[theorem]{Corollary}

\newtheorem{proposition}[theorem]{Proposition}

\theoremstyle{definition}
\newtheorem{definition}[theorem]{Definition}
\newtheorem{remark}[theorem]{Remark}

\newtheorem*{rem}{Remark}
\newtheorem*{prop}{Proposition}
\newtheorem*{exam}{Example}

\newtheoremstyle{thm1st}{\topsep}{\topsep}%
     {}
     {}
     {\bfseries}
     {}
     {.5em}
     {\thmname{#1}\thmnumber{ #2}\thmnote{ #3}}

  {}

\makeatletter
\newcommand{\Extend}[5]{\ext@arrow0099{\arrowfill@#1#2#3}{#4}{#5}}
\makeatother

\begin{document}
\title[Positive Scalar Curvature I]{contractible $3$-manifolds and positive scalar curvature (I)}

\author[Jian. Wang]{Jian Wang}
\address{Universit\"at Augsburg, Institut f\"ur Mathematik, Lehrstuhl f\"ur Differentialgeometrie, Universit\"atsstr. 14, 86159 Augsburg, Germany}

\address{Universit\'e Grenoble Alpes, Institut Fourier, 100 rue des maths, 38610 Gi\`eres, France}
\email{jian.wang.4@stonybrook.edu}

\maketitle
\begin{abstract} In this work we prove that the Whitehead manifold has no complete metric of non-negative scalar curvature.  This result can be generalized to the genus one case. Precisely, we show that no contractible genus one $3$-manifold admits a complete metric of non-negaitve scalar curvature. 
\end{abstract}

\section{Introduction}
If $(M^{n}, g)$ is a Riemannian manifold, its scalar curvature  is the sum of all sectional curvatures. A complete metric $g$ over $M$ is said to have positive scalar curvature (\emph{hereafter, psc}) if its scalar curvature is positive for all $x\in M$. There are many topological obstructions for a smooth manifold to have a complete \emph{psc} metric.

For example, from the proof of Thurston's geometrization conjecture in \cite{P1,P2,P3}, if $(M^{3}, g)$ is a compact 3-manifold with positive scalar curvature, then $M$ is a connected sum of some spherical 3-manifolds and some copies of $\mathbb{S}^{1}\times \mathbb{S}^{2}$. Recently, Bessi\`eres, Besson and Maillot \cite{BBM} obtained a similar result for (non-compact) $3$-manifolds with uniformly positive scalar curvature and bounded geometry. 

 A fundamental question is how to classify the topological structure of non-compact 3-manifolds with positive scalar curvature. Let us consider contractible 3-manifolds. For example, $\mathbb{R}^{3}$ admits a complete metric $g_{1}$ with positive scalar curvature, where  $$g_{1}=\sum_{k=1}^{3} (dx_{k})^{2}+(\sum_{k=1}^{3}x_{k}dx_{k})^{2}.$$
So far, $\mathbb{R}^{3}$ is the only known contractible 3-manifold which admits a complete \emph{psc} metric. This suggests the following question:\par
\vspace{2mm}
\noindent \textbf{Question:}
\emph{Is any complete contractible 3-manifold with positive scalar curvature homeomorphic to $\mathbb{R}^{3}$} ?\par
\vspace{2mm}
Note that for a compact $3$-manifold, its topological structure is fully determined by its homotopy type. However, it is known from \cite{Wh} that the topological structure of contractible 3-manifolds is much more complicated. For example, the Whitehead manifold (constructed in \cite{Wh}) is a contractible 3-manifold but not homeomorphic to $\mathbb{R}^{3}$.\par

In order to explain the construction of the Whitehead manifold, let us introduce the concept of a meridian. A meridian $\gamma\subset \partial N$ of the closed solid torus $N$ is an embedded circle which is nullhomotopic in $N$ but  not contractible in $\partial N$ (see Definition \ref{mer}). 

The Whitehead manifold is constructed from the Whitehead link. Recall that the Whitehead link is a link with two components illustrated in the below figure: 

\begin{figure}[H]
\centering{
\def\svgwidth{\columnwidth}
{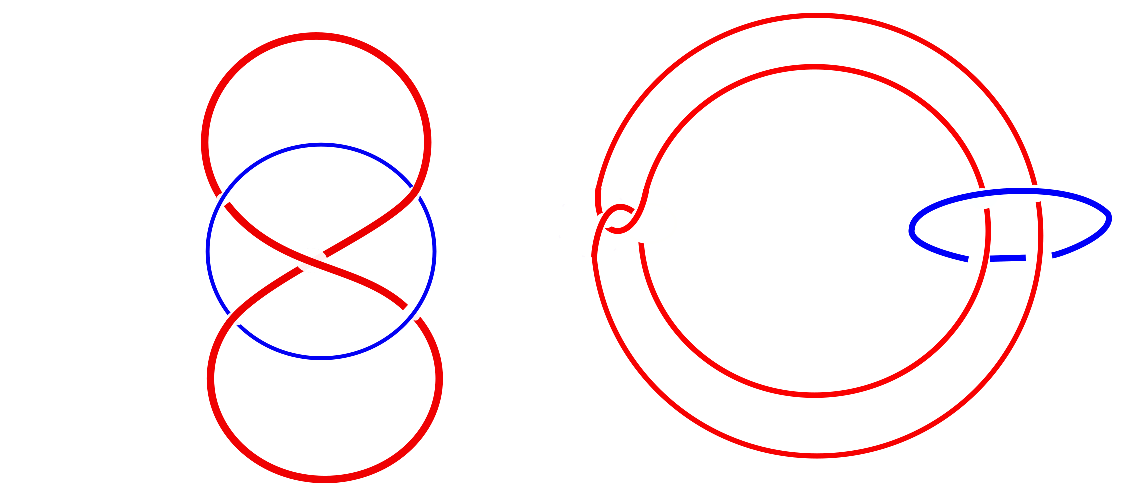}
\caption{}
\label{Fig1}
}
\end{figure}
 Choose a compact unknotted solid torus $T_{1}$ in $\mathbb{S}^{3}$. The closed complement of the solid torus inside $\mathbb{S}^{3}$ is another solid torus.  Take a second solid torus $T_{2}$ inside $T_{1}$ so that the core $K_{2}$ of $T_{2}$ (i.e. a deformation retraction of $T_2$) forms a Whitehead link with any meridian of the solid torus $T_{1}$. Similarly, $T_{2}$ is an unknotted solid torus. Then, embed $T_{3}$ inside $T_{2}$ in the same way as $T_{2}$ lies in $T_{1}$ and so on infinitely many times. Define $T_{\infty}:=\cap_{k=1}^{\infty} T_{k}$, called the Whitehead Continuum.

  The Whitehead manifold is defined as $Wh:=\mathbb{S}^{3}\setminus T_{\infty}$ which is an open $3$-manifold. 
\begin{figure}[H]
\centering{
\def\svgwidth{\columnwidth}
{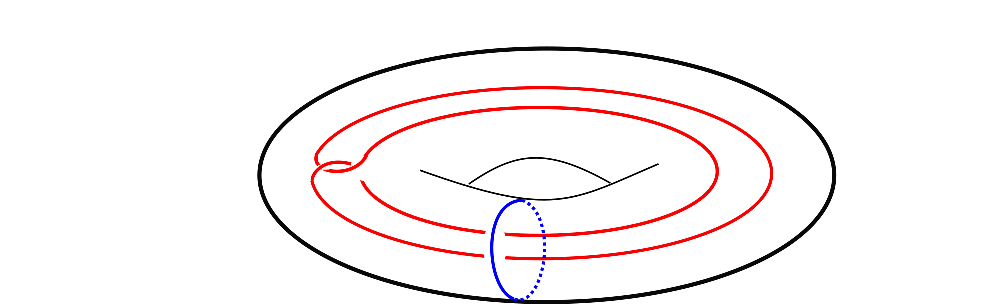}
\caption{}
\label{Fig2}
}
\end{figure}

Variation on the construction, like changing the knot at each step $k$, gives the family of so-called genus one $3$-manifolds, introduced in \cite{Mc} (see Definition \ref{G}). Remark that $\mathbb{R}^{3}$ is not genus one but genus zero, since it is an increasing union of $3$-balls. The construction can also be extended to the higher genus case \cite{Mc}.

An interesting question is whether the Whitehead manifold admits a complete metric with positive scalar curvature. In this article we answer it negatively:

\begin{theorem}\label{A}The  Whitehead manifold does not admit any complete metric with positive scalar curvature.\end{theorem}

The Whitehead manifold is the simplest example of contractible genus one $3$-manifolds. This result can be extended to the genus one case: 

\begin{theorem}\label{B} No contractible genus one $3$-manifold admits a complete metric with positive scalar curvature.
\end{theorem}

Further, we use the metric deformation in \cite{Kazdan} to have that: 

\begin{corollary}\label{C}
No contractible genus one $3$-manifold admits a complete metric with non-negative scalar curvature. 
\end{corollary}

\subsection{Topological properties of genus one $3$-manifolds} In this subsection, we use  the example of  Whitehead to illustrate  the topological properties of genus one $3$-manifolds. Based on the above construction, we use the symmetry of the Whitehead link to have that (see Section 2.3)
\begin{itemize}[leftmargin=15pt]
\item[(A)] the component $N_k:=\mathbb{S}^3\setminus T_k$ is a solid torus for each $k$;
\item[(B)] the Whitehead manifold $Wh$ is an increasing union of solid tori $\{N_k\}_k$;
\item[(C)] $N_{k-1}$ is located in $N_{k}$ in the same way as $T_1$ lies in $T_0$ (see Figure \ref{Fig3}). 
\end{itemize}

\begin{figure}[H]
\centering{
\def\svgwidth{\columnwidth}
{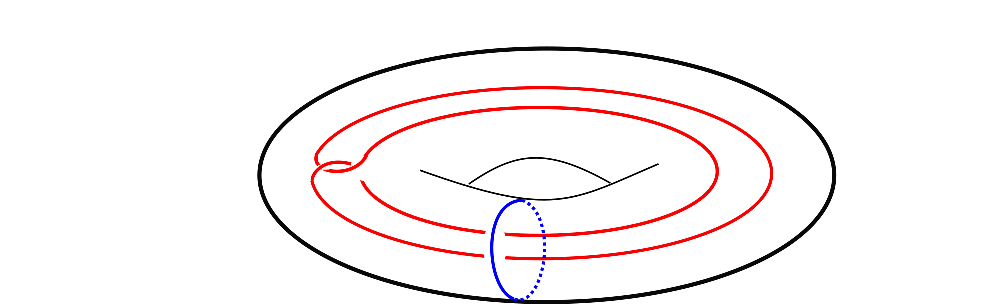}
\caption{}
\label{Fig3}
}
\end{figure}

The items (B) and (C) give an alternative construction of the Whitehead manifold. Variations on the construction such as changing the position of $N_{k-1}$ in $N_k$ at each step $k$ yield so-called genus one  $3$-manifolds, introduced by \cite{Mc}, referring to the fact that they are constructed by solid tori. 

\vspace{2mm}

In this article, we observe (see Theorem \ref{component}) that for any properly embedded plane $\Sigma\subset Wh$,  $k>0$ and any closed curve $\gamma\subset \partial N_{k}\cap \Sigma$, it holds one of the following:
\begin{enumerate}[leftmargin=15pt]
\item $\gamma$ is contractible in $\partial N_{k}$;
\item for $l<k$, $D\cap \text{Int}~N_{l}$ has at least $2^{k-l}$ components intersecting $N_{0}$, \end{enumerate} where $\gamma$ bounds a unique disc $D\subset \Sigma$ and $\{N_{k}\}_{k}$ is  defined in the item (A).

Such a topological property is called \emph{Property P}. 
Further, we show that a contractible genus one $3$-manifold satisfies Property P (see Theorem \ref{PT}). 

\vspace{1mm}

Property $P$ has a delicate connection with minimal surfaces. This connection plays a crucial role in the proof of Theorem \ref{B}.

\subsection{The effect of positive scalar curvature on minimal surfaces and minimal laminations}  We will use the following notations: 

\begin{itemize}[leftmargin=15pt]
\item  we let $(M, g)$ be a complete contractible genus one $3$-manifold, where $M$ is an increasing union of solid tori $\{N_k\}^\infty_{k=0}$.
\item  we define $\mathscr{L}=\cup_{t\in\Lambda }L_t$ to be a lamination in $(M, g)$ (i.e. a disjoint union of some embedded surfaces) whose each leaf is a complete (non-compact) stable minimal surface. (It may have infinitely many components.)
\end{itemize}

Minimal surfaces in a $3$-manifold with positive scalar curvature have been extensively studied in \cite{SY}, \cite{FS}, \cite{F} and many other works.
\begin{prop}
If $(M, g)$ has positive scalar curvature,  then one has that \begin{itemize}[leftmargin=17pt]
\item[(1)] each $L_t$ is homeomorphic to $\mathbb{R}^2$;
\item[(2)] $\int_{L_t}k(x)dv\leq 2\pi$ (see Theorem \ref{proper}),

\end{itemize}where $\kappa$ is the scalar curvature of $(M, g)$ and $dv$ is the volume form of the induced metric on $L_t$. 
\end{prop}

The item (1) was firstly pointed out by Schoen and Yau \cite{SY}. The item (2) is the so-called \emph{extrinsic Cohn-Vossen's inequality}.  The original Cohn-Vossen's inequality \cite{CV} is the same for a complete surface with positive curvature, where $\kappa$ is replaced by the sectional curvature. 
\vspace{1mm}

All these results reveal a classical fact: the positivity of scalar curvature  constrains the topological and geometric structures of stable minimal surfaces.  This fact also appears in the works of Schoen and Yau \cite{SY3, SY1, SY} as well as Gromov-Lawson's \cite{GL} and various other authors. 
\vspace{2mm}

We now explain the interface between the topological properties (especially Property P) and the scalar curvature. First, the extrinsic Cohn-Vossen's inequality and Property P imply that 

\vspace{1mm}

\begin{lemma}\label{trivial} If $(M, g)$ has positive scalar curvature, then there is an integer $k_0$ so that for each $k\geq k_0$ and $t\in \Lambda$, any circle in $L_t\cap \partial N_k$ is contractible in $\partial N_k$. \end{lemma}

\vspace{1mm}

Then, we show the following theorem: 

\begin{theorem}\label{D} If $(M, g)$ has positive scalar curvature, there is an integer $k_0>0$ so that for $k\geq k_0$, $\mathscr{L}\cap \p N_k(\epsilon)$ is contained in a disjoint union of \textbf{finitely many} discs in $\p N_k(\epsilon)$, where $N_{k}(\epsilon):=N_{k}\setminus N_{\epsilon}(\partial N_{k})$ and $N_{\epsilon}(\partial N_{k})$ is some tubular neighborhood of $\partial N_{k}$.\end{theorem}

In the case when $\mathscr{L}$ has finitely many components, Theorem \ref{D} can be easily deduced from {Lemma} \ref{trivial}. Generally, we require a finiteness lemma to complete the proof. 

\subsection{Rough idea of the proof for the Whitehead case}  The proof of Theorem \ref{B} is the same as Theorem \ref{A}. In the following, we just explain the idea for the Whitehead manifold.

We argue by contradiction. Suppose  the contrary that $(Wh, g)$ is complete with positive scalar curvature. It is an increasing union of solid tori $\{N_k\}_k$ (see Sections 1.1 and 2.3). 

Choose a meridian $\gamma_k\subset \partial N_k$ of $N_k$ (see Definition \ref{mer}). Roughly speaking, it is spanned by an embedded stable minimal disc $\Omega_{k}$ in $N_{k}$. Its existence is ensured by the work of Meeks and Yau (see \cite{YM, YM1} or Theorem 6.28 of \cite{CM1}), when $\partial N_{k}$ is mean convex.

\vspace{1mm}

 The sequence $\{\Omega_k\}_k$ sub-converges towards a minimal lamination $\mathscr{L}:=\cup_{t\in \Lambda}L_t$ (a disjoint union of embedded minimal surfaces) (see Theorem \ref{limits}). It may have infinitely many components.  Each leaf is a complete (non-compact) stable minimal surface. 
 
 \vspace{1mm}
 
 For our convenience, we assume that the sequence $\{\Omega_k\}$ converges to $\mathscr{L}$. 
 
 \vspace{1mm}
 
 Theorem \ref{D} shows that $\mathscr{L}\cap \partial N_{k_0}(\epsilon)$ is contained in a disjoint union of finitely many discs $\{D_j\}$ in $\partial N_{k_0}(\epsilon)$, where $N_{k_0}(\epsilon)$ is defined in Theorem \ref{D}.  Then, we can find a longitude $\theta\subset \partial N_{k_0} (\epsilon) \setminus\amalg_j D_j$ of $N_{k_0}(\epsilon)$ ( i.e. an embedded circle in $\partial N_{k_0}(\epsilon)$ whose intersection number with any meridian is $\pm 1$). 
 
 The convergence of $\{\Omega_k\}$ implies that $\Omega_k$ is away from $\theta$ for $k$ large enough. However, each $\Omega_k$ contains a meridian $\alpha_k$ of $N_{k_0}(\epsilon)$ for $k>k_0$ (see Lemma \ref{one}),  which contradicts the fact that the intersection number of $\alpha_k$ and $\theta$ is $\pm 1$. 
 
 \vspace{2mm}

If $\partial N_{k}$ is not mean convex, we modify the metric in a small neighborhood of $\partial N_{k}$ so that for this new metric it becomes mean convex. Then $\Omega_{k}$ is minimal for this new metric and  for the original one away from the neighborhood of $\partial N_{k}$ (near $N_{k-1}$, for example), which is sufficient for our proof. 

\subsection{Main Difficulties in the proof}  In our proof, we use minimal surfaces to detect the topology of $3$-manifolds. However, we counter several difficulties about minimal surfaces.  

\subsubsection{Issue about convergence } For the sequence of minimal surfaces $\{\Omega_{k}\}$ as mentioned above,  there is no available uniform local area bound (see Section 5.2), which is necessary and crucial when considering the convergence of such a sequence in the classical sense. To overcome this obstacle, we then consider the convergence towards a lamination (see \cite{CM}). 

However, this approach may lead to another problem that the limit may have infinitely many components. We will explain how to deal with it later.  

\subsubsection{Difficulty about the properness of minimal surfaces} In the proof, we always require to deform $\partial N_k$ so that it intersects some leaf $L_t$ transversally, where $L_t$ and $N_k$ are defined in Section 1.2. The existence of such a deformation is fully determined by the properness of $L_t$. 

In the article, we obverse that the \emph{extrinsic Cohn-Vossen's inequality} shows that $L_t$ is properly embedded for each $t\in \Lambda$ (see Theorem \ref{proper}).  Thus, the required deformation always exists.

\subsubsection{Issue on Finiteness} In the proof of Theorem \ref{D}, the lamination $\mathscr{L}$ may have infinitely many components. We now explain the idea of our approach. 

Let $\mathscr{L}$, $(M, g)$ and $\{N_k\}$ be assumed as in Section 1.2. Assume that $(M, g)$ has positive scalar curvature. Lemma \ref{trivial} gives an integer $k_0$ such that for $k\geq k_0$ and $t\in \Lambda$, any circle in $\partial N_k\cap L_t$ is contractible in $\partial N_k$.  
\vspace{1mm}

We will work on an open solid torus $\text{Int}~ N_k$ for a fixed $k\geq k_0$. Let $\Sigma^i_t$ be one component of $L_t\cap \text{Int} ~N_k$, where $i\in I_t$. Each component of the boundary $\partial \Sigma^i_t$ is a null-homotopic circle in $\partial N_k$. 
We show that (see Lemma \ref{existence})

\begin{itemize}[leftmargin=10pt]
\item $\Sigma^i_t$ cuts $\text{Int}~ N_k$ into two components; 
\item there is a unique component $B^{i}_t$ satisfying  that $\pi_1(B^i_j)\rightarrow \pi_1(N_k)$ is trivial. 
\end{itemize} Further, there is a partially ordered relation on $\{B^i_t\}_{t\in \Lambda, i\in I_t}$ induced by the inclusion (see Lemma \ref{sep}).

The set $(\{B^i_t\}, \subset)$ may be an infinite set. However, roughly speaking, we use the properties of minimal laminations to show that the set $S$ of the maximal elements in $(\{B^i_t\}, \subset)$ is finite (see the finiteness lemma in Section 6.3). 

Because of the maximality, it is sufficient for the proof to study the minimal surfaces associated with an element in $S$.  The resting argument is the same as the finite case.

\subsection{Organization of the paper} 
  
When applying the minimal hypersuface technique to a contractible $3$-manifold, we must combine its topological property and its geometric structure. The plan of this paper is as follows. 
  
 For the first part of the paper, we describe the topological properties of contractible $3$-manifolds. In Section 2, we begin with related background on knot theory and $3$-manifolds, such as the geometric index and genus one $3$-manifolds.  In Section 3, we focus on embedded discs in the Whitehead manifold. An interesting fact is that the behavior of these embedded discs is related to the geometric index. Their relation is clarified by  Theorem \ref{component}. Based on this relation, we introduce a new topological property, called Property $P$ and we show that any contractible genus one $3$-manifold satisfies this property in Section 4.  

In the second part, we treat minimal surfaces and related questions. In Section 5, we first describe the construction of embedded discs $\{\Omega_k\}$ in a complete genus one $3$-manifold and then study the properties of the limit surfaces. In Section 6, we explain the interplay between Property P and the positivity of scalar curvature.

For the third part, we give the proofs of the main theorems . In Section 7, we first use Theorem \ref{D} to complete  proof of Theorem \ref{B}. Then, we apply the metric deformation of Kazdan and give a complete proof of Corollary \ref{C}.

\section{Background}

\subsection{Simply connected at infinity}

\begin{definition}A topological space $M$ is  \emph{simply connected at infinity} if for any compact set $K\subset M$, there exists a  compact set $K'$ containing $K$ so that the induced map $\pi_{1}(M\setminus K')\rightarrow \pi_{1}(M\setminus K)$ is trivial. \end{definition}

For example, the Whitehead manifold is not simply-connected at infinity (see \cite{Wh}). 

The Poincar\'e conjecture \cite{P1, P2, P3} tells that any contractible $3$-manifold is irreducible (i.e. any embedded $2$-sphere in the $3$-manifold bounds a closed $3$-ball).   A classical theorem \cite{HP, S} implies that the only contractible and simply-connected at infinity 3-manifold is $\mathbb{R}^{3}$. As a consequence, the Whitehead manifold is not homeomorphic to $\mathbb{R}^{3}$.

\subsection{Knots basics}
In this part, we focus on Knot theory in a solid torus. We begin with a meridian  of the solid torus. In the following, we suppose that $N$ is a closed solid torus. 
\begin{definition}\label{mer} 
An embedded circle $\gamma \subset \partial N$ is called a meridian of the closed solid torus $N$ if $\gamma$ is not contractible in $\partial N$ but is nullhomotopic in $N$.

 An embedded disc $(D, \partial D)\subset (N, \partial N)$ is a meridian disc if $\partial D$ is a meridian of $N$.

\end{definition}

\begin{remark}
The kernel of the induced map $\pi_{1}(\partial N)\rightarrow \pi_{1}(N)$ is isomorphic to $\mathbb{Z}$. Each meridian $\gamma$ of $N$ belongs to the kernel. Since $\gamma$ is an embedded circle, it is a generator of the kernel. Therefore, the meridians of $N$ are unique, up to homotopy and orientation. 

In addition, there is an embedded circle $\theta\subset \partial N$ so that $[\gamma]$ and $[\theta]$ generate $H_{1}(\partial N)$, where $\gamma$ is a meridian of $N$.

\end{remark}

\begin{definition}\label{long}

An embedded circle $\theta \subset \partial N$ is called a longitude  of $N$ if, for any meridian $\gamma$, $[\gamma]$ and $[\theta]$ generate $H_{1}(\partial N, \mathbb{Z})$.
 
A knot $K\subset \text{Int} ~N$ is called a trivial knot in the solid torus $N$ if there exists  an embedded disc $D\subset \text{Int}~N$ with boundary $K$.
\end{definition}
Remark that, as a knot, any longitude is isotopic to the core of $N$ (i.e. a deformation retraction of $N$) in $N$.

\begin{definition} Let $N'\subset N$ be two solid tori and $K\subset N'$ be  a core of $N'$ (i.e. a deformation retraction of $N'$). The core $K$ is an embedded  closed curve  in $N$.  

The solid torus $N'$ is homotopically trivial in $N$ if the core $K$ is contractible in $N$. The solid torus $N'$ is unknotted in $N$ if the core $K$ bounds an embedded disc in $N$. 
\end{definition}

Remark that any unknotted solid torus in $N$ is homotopically trivial in $N$. Conversely, it may not  hold. For example, in Whitehead manifold, $N_{k-1}$ is homotopically trivial in $N_k$ but is unknotted in $N_k$ (see Figure \ref{Fig3}).

\vspace{2mm}

In the following, let us consider two closed solid tori $N'$ and $N$ satisfying that (1)$N' \subset \text{Int}~N$; (2) $N'$ is homotopically trivial in $N$. We use the Mayer-Vietoris sequence to show 
\begin{lemma}\label{MV}If the closed solid torus $N'\subset \text{Int}~N$ is homotopically trivial  in the closed solid torus $N$, then $H_{1}(\overline{N\setminus N'})=\mathbb{Z}^{2}$ and the kernel of the induced map $H_{1}(\partial N')\rightarrow H_{1}(\overline{N\setminus N'})$ is generated by a longitude of $N'$. 
\end{lemma}
\begin{proof} The Mayer-Vietoris sequence gives the long exact sequence: 
\begin{equation*}
H_{2}(N)\rightarrow H_{1}(\partial N')\rightarrow H_{1}({N'})\oplus H_{1}(\overline{N\setminus N'})\rightarrow H_{1}(N)\rightarrow \hat{H}_{0}(\partial N').
\end{equation*}
Since $H_{2}(N)$ and $\hat{H}_{0}(\partial N')$ are both trivial, then $H_{1}(\overline{N\setminus N'})\cong \mathbb{Z}^{2}$ and it is generated by any meridian $\gamma'$ of $N'$ and a longitude of $N$.

In addition, the image of the map $H_{1}(\partial N')\rightarrow H_{1}(\overline{N\setminus N'})$ is a subgroup of rank one which is generated by the meridian $\gamma'$ of $N'$. 
The kernel of $H_{1}(\partial N')\rightarrow H_{1}(\overline{N\setminus N'})$ is also of rank one and generated by $[\theta']$, where the circle $\theta'\subset \partial N'$ is an embedded cirlce. 
 Therefore, $H_{1}(\partial N')$ is generated by $[\gamma']$ and $[\theta']$. By Definition \ref{long}, $\theta'$ is a longitude of $N'$. That is to say, the longitude $\theta'$ is a generator of the kernel of $H_{1}(\partial N')\rightarrow H_{1}(\overline{N\setminus N'})$.
\end{proof}

\subsection{Properties of the Whitehead manifold}
As in Introduction, the Whitehead manifold is defined as $Wh:=\mathbb{S}^{3}\setminus \cap_{k=0}^{\infty} T_{k}$. Let $N_{k}$ be the complement of $T_{k}$, an closed solid torus. 
\begin{remark}\label{setup in Wh}
The Whitehead manifold has the following properties:

\begin{enumerate}[leftmargin=15pt]
\item $Wh$ can be written as an increasing union of $\{N_{k}\}$. 
\item The core $K_{k}$ of $N_{k}$ (i.e. a deformation retraction of $N_k$) is a non-trivial knot in the solid torus $N_{k+1}$. Furthermore, the link $K_{k}\amalg\gamma_{k+1}$ is a Whitehead link for each meridian $\gamma_{k+1}$ of $N_{k+1}$. This is a consequence of the symmetry of the Whitehead link.
\item Each $K_{k}$ is unknotted in $\mathbb{S}^{3}$. For each $j>k$,  $K_{k}$ is nullhomotopic in $N_{j}$ but is a non-trivial knot in $N_{j}$.
 \end{enumerate}
 \end{remark}

\subsection{Geometric Index}
\begin{definition}\cite{Sch}\label{index} If $N'\subset\text{Int} ~N$ are two solid tori, the \emph{geometric index} of $N'$ in $N$, $I(N', N)$, is the minimal number of points of the intersection of the core of $N'$ (i.e. a deformation retraction of $N'$) with a meridian disk of $N$.
\end{definition}
For example, in $Wh$, the geometric index $I(N_{k}, N_{k+1})=2$ for each $k$, where $N_{k}$ is illustrated as in Sections 1.1 and 2.3. 

See \cite{Sch} for the following results about the geometric index.
\begin{itemize}[leftmargin=15pt]

 \item Let $N_{0}$, $N_{1}$, and $N_{2}$ be solid tori so that $N_{0}\subset \text{Int}~N_{1}$ and $N_{1}\subset \text{Int} ~N_{2}$. Then $I(N_{0}, N_{2})=I(N_{0}, N_{1})I(N_{1}, N_{2})$. 
\item If $N_{0}$ and $N_{1}$ are unknotted solid tori in $\mathbb{S}^{3}$ with $N_{0}\subset \text{Int}~ N_{1}$, and if $N_{0}$ is homotopically trivial in $N_{1}$, then $I(N_{0}, N_{1})$ is even.
\end{itemize}

 In the following, the loop lemma will be used several times. 

\begin{lemma}\label{loop} \textnormal{(See [Theorem 3.1, Page 54] in \cite{HA})} Let $M$ be a $3$-manifold with boundary $\partial M$, not necessarily compact or orientable. If there is a map $f: (D^{2}, \partial D^{2})\rightarrow (M, \partial M)$ with the property that $f|_{\partial D^{2}}$ is not nullhomotopic in $\partial M$, then there is an embedding with the same property. 
\end{lemma}

\begin{corollary}\label{unknotted} Let $N'$ and $N$ be two closed solid tori with $N'\subset \text{Int}~N$. If the core $K$ of $N'$ is a trivial knot in $N$, then there is a meridian disc $(D, \partial D)\subset (N \setminus N', \partial N)$. Moreover, $I(N', N)$ is equal to zero. 
\end{corollary}
 \begin{proof} Let $D_{0}\subset \text{Int} ~N$ be the embedded disc with boundary $K$ and $\gamma\subset \partial N$ a meridian of $N$. Van-Kampen's Theorem shows that $\pi_{1}(N\setminus D_{0})$ is  isomorphic to $\pi_{1}(N)$. Thus, $\gamma$ is nullhomotopic in $N\setminus D_{0}$. It is also contractible in $N\setminus K$. 
 
 The inclusion map $\pi_{1}(N\setminus N')\rightarrow \pi_{1}(N\setminus K)$  is an isomorphism. Therefore, $\gamma$ is nullhomotopic in $N\setminus N'$ but not contractible in $\partial N$. Lemma \ref{loop} allows us to find an embedded disc $(D, \partial D)\subset (N\setminus N', \partial N)$. The boundary $\partial D$ is a meridian. The disc $D$ is the required candidate in the assertion. 
 
By Definition \ref{index}, $I(N', N)$ is equal to zero.     
 \end{proof}

In the following, let us consider two solid tori $N'\subset \text{Int}~N$    satisfying 1) $N'$ is homotopically trivial in $N$; 2) the geometric index $I(N', N)>0$.

\begin{lemma}\label{injective} Suppose that the closed solid torus $N'\subset \text{Int} ~N$ is homotopically trivial in the closed solid torus $N$. If $I(N', N)>0$, then the two induced maps $i_{1}:\pi_{1}(\partial N)\rightarrow \pi_{1}({N\setminus N'})$ and $i_{2}:\pi_{1}(\partial N')\rightarrow \pi_{1}(\overline{N\setminus N'})$ are both injective. 
\end{lemma}

\begin{proof}
Suppose that either $i_{1}$ or $i_{2}$ is not injective. 

If $i_{1}$ is not injective, we use Lemma \ref{loop} to find an embedded disc $(D_{1}, \partial D_{1})\subset (N\setminus N', \partial N)$ where $\gamma_{1}:=\partial D_{1}$  is not nullhomotopic in $\partial N$. It is a meridian disc of $N$ with $D_{1}\cap N'=\emptyset$ (see Definition \ref{mer}). Therefore, the index $I(N', N)$ is equal to zero, which contradicts our assumption that $I(N', N)>0$.

If $i_{2}$ is not injective, Lemma \ref{loop} gives an embedded disc $( D_{2}, \partial D_{2})\subset (\overline{N\setminus N'}, \partial N')$. The embedded circle $\theta:=\partial D_{2}$ is not contractible in $\partial N'$. 

Since $\theta$ bounds an embedded disc $D_{2}\subset N\setminus N'$, it is a trivial knot in $N$. Furthermore, $[\theta]$  belongs to the kernel of the map $H_{1}(\partial N')\rightarrow H_{1}(\overline{N\setminus N'})$. From Lemma \ref{MV}, $\theta$ is a longitude of $N'$.

  Recall that as a knot, any longitude of $N'$ is isotopic to the core $K$ of $N'$. 
  Therefore, $K$ is isotopic to $\theta$ and  a trivial knot in $N$. 
  From Corollary \ref{unknotted}, $I(N', N)=0$. This is a contradiction and finishes the proof.   
\end{proof}

We now introduce some notations about the circles in a disc.
\begin{definition}\label{outmost}Let $C:=\{c_{i}\}_{i\in I}$ be a finite set of pairwise disjoint circles in the disc $D^{2}$ and $D_{i}\subset D^{2}$ the unique disc with boundary $c_{i}$. Consider the set $\{D_{i}\}_{i\in I}$ and define a partially ordered relation induced by the inclusion. $(\{D_{i}\}_{i\in I}, \subset)$ is a partially ordered set. For each maximal element $D_{j}$ in $(\{D_{i}\}_{i\in I}, \subset)$, its boundary $c_{j}$ is defined as a \emph{maximal circle} in $C$. For each minimal element $D_{j}$, its boundary $c_{j}$ is called a \emph{minimal circle} in $C$.
\end{definition}

\begin{lemma}\label{one}Suppose that the closed solid torus $N'\subset \text{Int}~N$ is homotopically trivial in the closed solid torus $N$. If $I(N', N)>0$, then any meridian disc $D$ of $N$ contains a meridian of $N'$.
\end{lemma}

\begin{proof}
 We may assume  that the closed meridian disc $D$ intersects $\partial N'$ transversally. The set $D\cap \partial N'$ is a disjoint union of circles $\{c_{i}\}_{i\in I}$.  Each $c_{i}$ bounds a unique closed disc $D_{i}\subset \text{Int}~ D$.
 
  Consider the set $C^{non}:=\{c_{i}~|~ c_{i}~ \text{is not contractible in }~ \partial N'\}$ and the set $C^{max}=\{c_{i}|~c_{i}$ is a maximal circle in $\{c_{i}\}_{i\in I}\}$.

We will show that $C^{non}$ is nonempty and a minimal circle in $C^{non}$ is a desired meridian.

 Suppose the contrary that $C^{non}$ is empty. Therefore, each $c_{i}\in C^{max}$ is contractible in $\partial N'$ and bounds a disc $D'_{i} \subset \partial N'$.
Consider the immersed disc $$\hat{D}:=(D\setminus\cup_{c_{i}\in C^{max}}D_{i})\cup(\cup_{c_{i}\in C^{max}}D'_{i})$$ with boundary $\gamma$ ($\gamma$ is a meridian of $N$). Then, $\hat{D}\cap \text{Int} ~N'=\emptyset$, which implies that  $\gamma$ is contractible in $\overline{N\setminus N'}$. 

However, Lemma \ref{injective} shows that $\pi_{1}(\partial N)\rightarrow \pi_{1}(\overline{N\setminus N'})$ is injective. Namely, $[\gamma]=1$ in $\pi_{1}(\partial N)$, which is in contradiction with  our hypothesis that $\gamma$ is a meridian. We can conclude that $C^{non}\neq\emptyset$.

\vspace{2mm}

In the following, we will prove that each minimal circle $c_{j}$ in $C^{non}$ is a required meridian. By Definition \ref{mer}, it is sufficient to show that $c_{j}$ is contractible in $N'$.  Our method is to construct an immersed disc $\hat{D}_{j}\subset {N'}$ with boundary $c_{j}$. 

Let us consider the set $C_{j}:=\{c_{i}~| c_{i}\subset \text{Int} ~D_{j}$ for $i\in I\}$ and  the subset $C_{j}^{max}$ of maximal circles in $C_{j}$. We now have two cases:  $C_{j}=\emptyset$ or $C_{j}\neq \emptyset$. 

\textbf{Case I:} If $C_{j}$ is empty, we consider the set $Z:=\text{Int}~D_{j}$ and define the disc $\hat{D}_{j}$ as $\text{Int}~D_{j}$. 

\textbf{Case II:} If $C_{j}$ is not empty, then $C^{max}_{j}$ is also nonempty. 
From the minimality of $c_{j}$ in $C^{non}$, each $c_{i}\in C^{max}_{j}$ is nullhomotopic in $\partial N'$ and bounds a disc $D''_{i}\subset\partial N'$. 

Define  the set $Z:=\text{Int} ~D_{j}\setminus \cup_{c_{i}\in C^{max}_{j}}{D}_{i}$ and the new disc $\hat{D}_{j}:=Z\cup(\cup_{c_{i}\in C^{max}_{j}} D''_{i})$ with boundary $c_{j}$. 

\vspace{2mm}

Let us explain why the disc $\hat{D}_{j}$ is contained in ${N'}$. In any case, $\partial N'$ cuts $N$ into two connected components, $\text{Int} ~N'$ and $N\setminus N'$. The set $Z$ is one of these components of $ \text{Int}~D_{j}\setminus \partial N'$. Therefore it must be contained in $\text{Int}~ N'$ or $N\setminus N'$.

If $Z$ is in $N\setminus N'$, then the disc $\hat{D}_{j}$  is contained in $\overline{N\setminus N'}$. Thus $c_{j}$ is contractible in $\overline{N\setminus N'}$. However, Lemma \ref{injective} gives that  the induced map $\pi_{1}(\partial N')\rightarrow\pi_{1}(\overline{N\setminus N'})$ is injective. That is to say, $c_{j}$ is homotopically trivial in $\partial N'$, which contradicts the choice of $c_{j}\in C^{non}$. We can conclude that $Z$ is contained in $\text{Int}~N'$.

Therefore, $\hat{D}_{j}$ is contained in ${N'}$. Namely, $c_{j}$ is null-homotopic in ${N'}$. From Definition \ref{mer}, we can conclude that $c_{j}\subset D$ is a meridian of $N'$. This finishes the proof. 
\end{proof}

\subsection{Genus one $3$-manifolds}(see \cite{GRW})Let us describe McMillan's construction in \cite{Mc}.
\begin{definition}\label{G}(Genus one 3-manifold) A \emph{genus one 3-manifold} $M$ is the ascending union of solid tori $\{N_{k}\}_{k\in \mathbb{N}}$, so that for each $k$, $N_{k}\subset \text{Int}~N_{k+1}$ and the geometric index of $N_{k}$ in $N_{k+1}$ is not equal to zero. 
\end{definition}

\begin{theorem}\label{genus one}\textnormal{(See [Theorem 2.8, Page 2042] in \cite{GRW} )}~~~~~~~~~~~~~~~~~~~~~~~~~~~~~~
\begin{enumerate}[leftmargin=15pt]
\item A genus one 3-manifold defined with a sequence of closed solid tori $\{N_{k}\}_{k\in \mathbb{N}}$ so that each $N_{k}$ is contractible in $N_{k+1}$, is a contractible 3-manifold that is not homeomorphic to $\mathbb{R}^{3}$. 
\item Any contractible genus one 3-manifold can be written as  an ascending union of solid tori $\{N_{k}\}_{k\in \mathbb{N}}$ so that $N_{k}$ is contractible in $N_{k+1}$.
\end{enumerate}
\end{theorem}
For example, the Whitehead manifold is a contractible  genus one 3-manifold. \par
A contractible  genus one 3-manifold $M:=\cup_{i=0}^{\infty} N_{k}$ satisfies the following: 
\begin{enumerate}[leftmargin=15pt]
\item For each $k$, $N_{k}$ is homotopically trivial in $N_{k+1}$. Moreover, $I(N_{k}, N_{k+1})\geq 2$. 
\item For each $j>k$, the core $K_{k}$ of $N_{k}$ is null-homotopic in $N_{j}$ but a nontrivial knot in $N_{j}$.
\item If $N_{j}$ is viewed as an unknotted solid torus in $\mathbb{S}^{3}$, then the curves  $K_{k}$ and $\gamma_{j}\subset \mathbb{S}^{3}$ are linked in $\mathbb{S}^{3}$, for each meridian $\gamma_{j}$ of $N_{j}$ for $j>k$. Moreover, its linking number is zero. 
\item However, the knot $K_{k}\subset\mathbb{S}^{3}$ may be knotted in $\mathbb{S}^{3}$.

\end{enumerate}

\section{Embedding disc(s) in the Whitehead manifold}
As in Section 2.3, ${Wh}\subset \mathbb{S}^{3}$ is an increasing union of closed solid tori $\{N_{k}\}^{\infty}_{k=0}$ so that the geometric index $I(N_{k}, N_{k+1})=2$, for each $k$. For any $j>k$, the core $K_{k}$ of $N_{k}$ is a non-trivial knot in $N_{j}$ but is unknotted in $\mathbb{S}^{3}$. In addition, the curves $K_{k}$ and $\gamma_{j}$ are linked with zero linking number, for any meridian $\gamma_{j}$ of $N_{j}$. In this section, we investigate some embedded discs in the Whitehead manifold.
 
\begin{lemma}\label{meridian}
 Any embedded circle $\gamma \subset \partial N_{k}$ which is the boundary of a closed embedded disc $D$ in $\text{Wh}$ but is not nullhomotopic  in $\partial N_{k}$,  is a meridian of $ N_{k}$.
\end{lemma}

\begin{proof} The compactness of  $D$  gives an integer $k_{0}>k$ so that $D$ is contained in $N_{k_{0}}$.

Let $\gamma$ belong to the homology class $p[\gamma_{k}]+q[\theta_{k}]$ in $H_{1}(\partial N_{k})$, where $\gamma_{k}$ and $\theta_{k}$ are a meridian and a longitude of $N_{k}$. 
Since $N_{k}$ is an unknotted solid torus in $\mathbb{S}^{3}$, $\gamma$ (as a knot in $\mathbb{S}^{3}$) is isotopic to the torus knot $K_{p, q}$ in $\mathbb{S}^{3}$. \par
Because the knot $\gamma$ bounds an embedded disc $D$ in $N_{k_{0}}$, it is a trivial knot in $N_{k_{0}}\subset \mathbb{S}^{3}$. Hence, $\gamma$ is unknotted in $\mathbb{S}^{3}$. A result (see Page 53 of  \cite{Rol}) shows that $p=\pm 1$ or $q=\pm1$.

Since the knot $\gamma$ is trivial in $N_{k_{0}}$, we use Corollary \ref{unknotted} to find a meridian disc $(D_{1}, \partial D_{1})\subset ({N}_{k_{0}}, \partial N_{k_{0}})$ with $D_{1}\cap \gamma=\emptyset$ . Because the geometric index $I(N_{k}, N_{k_{0}})>0$, it contains a meridian ${\gamma'}_{k}$ of $\partial N_{k}$ (Lemma \ref{one}).

 Therefore, ${\gamma'}_{k}\cap \gamma$ is empty. Their intersection number on $\partial N_{k}$ must be zero.\par

The intersection number of $\gamma$ and ${\gamma'}_{k}$ is $q$, which implies that $p=\pm 1, q=0$. That is to say, $\gamma$ is homotopic to the meridian $\gamma'_{k}$ on $\partial N_{k}$. This completes the proof.  
\end{proof}

\begin{theorem}\label{component}
Any $\gamma\subset \partial N_{k}$ bounding an embedded disc $D$ in $\text{Wh}$ satisfies one of the following:
\begin{enumerate}
\item $[\gamma]$ is trivial in $\pi_{1}(\partial N_{k})$,
\item $D \cap \text{Int}~N_{l}$ has at least $I(N_{l}, N_{k})$ components intersecting $N_{0}$, for $l<k$.
\end{enumerate}
\end{theorem}
Note that the geometric index $I(N_{l}, N_{k})$ is equal to $2^{k-l}$.
\begin{proof}
We argue by induction on $k$. 
\begin{itemize}
\item When $k=0$, it is trivial.
\item We suppose that it holds for $N_{k-1}$.
\end{itemize} 
We suppose that $\gamma$ is not contractible in $\partial N_{k}$. Lemma \ref{meridian} tells that $\gamma$ is a meridian of $N_{k}$.
 In addition,  the linking number of $\gamma\amalg K_{k-1}$ is zero, where $K_{k-1}$ is the core of $N_{k-1}$.\par

We may assume that $D$ intersects $\partial N_{k-1}$ transversally. The set $D\cap \partial N_{k-1}$ has finitely many components $C:=\{\gamma_{i}\}_{i\in I}$. Each component $\gamma_{i}$ is an embedded circle and bounds a unique closed disc $D_{i}\subset \text{Int}~D$. 

Let $\{\gamma_{j}\}_{j\in I_{0}}$ be the set of maximal circles in $C$ where $I_{0}\subset I$. Each $\gamma_{j}$ is the boundary of the disc $D_{j}$, for $j\in I_{0}$.

\vspace{2mm}

\noindent $\bold{Claim}$: The set $\{\gamma_{j}\}_{j\in I_{0}}$ contains (at least)  two meridians of ${N}_{k-1}$.

By Lemma \ref{injective},  the maps $\pi_{1}(\partial N_{k})\rightarrow \pi_{1}(\overline{N_{k}\setminus N_{0}})$ and $\pi_{1}(\partial N_{k})\rightarrow \pi_{1}(\overline{N_{m}\setminus N_{k}})$ are both injective for any $m>k$. Van-Kampen's Theorem gives an isomorphism between $\pi_{1}(\overline{N_{m}\setminus N_{0}})$ and $\pi_{1}(\overline{N_{k}\setminus N_{0}})\ast_{\pi_{1}(\partial N_{k})} \pi_{1}(\overline{N_{m}\setminus N_{k}})$. We use  [Theorem 11.67, Page 404] in \cite{R} to see that the map $\pi_{1}(\partial N_{k})\rightarrow \pi_{1}(\overline{N_{m}\setminus N_{0}})$ is also injective for $m>k$.  Therefore, since $\gamma$ is not contractible  in $\partial N_{k}$, we can conclude that it is not contractible in $\overline{Wh\setminus N_{0}}$.

If $\gamma_{j}$ is homotopically trivial in $\partial N_{k-1}$ for each $j \in I_{0}$, then one finds a disc $D'_{j}\subset \partial N_{k-1}$. Consider a new disc $D':=(\Sigma\setminus\cup_{j\in I_{0}}D_{j})\cup(\cup_{j\in I_{0}}D'_{j})$ in $\overline{Wh\setminus N_{0}}$ with boundary $\gamma$. Hence, $\gamma$ is contractible in $\overline{Wh\setminus N_{0}}$. This contradicts with the last paragraph.  We can conclude  that  one of circles $\{\gamma_{j}\}_{j\in I_{0}}$ is non-contractible in $\partial N_{k-1}$.  
 Further, Lemma \ref{meridian} implies that there is at least one meridian of ${N}_{k-1}$ in $\{\gamma_{j}\}_{j\in I_{0}}$.

\vspace{1mm}
  In the following, we argue by contradiction. Suppose that there is a unique meridian of ${N}_{k-1}$ in the set $\{\gamma_{j}\}_{j\in I_{0}}$. That is to say, there is a unique $j_{0}\in I_{0}$ such that $\gamma_{j_{0}}$ is a meridian of $N_{k-1}$.  Remark that each $\gamma_{j}$ bounds the unique disc $D_{j}\subset D$.

 If $\gamma_{j}$ is not contractible in $\partial N_{k}$ for some $j\in I_{0}\setminus\{j_{0}\}$, Lemma \ref{meridian} shows that it is a meridian, which contradicts the uniqueness of $j_{0}$. Therefore, $\gamma_{j}$ is nullhomotopic in $\partial N_{k-1}$, for each $j\in I_{0}\setminus \{j_{0}\}$.

Consider a meridian disc $\hat{D}_{j_{0}}$ of $N_{k-1}$ with boundary $\gamma_{j_{0}}$, which intersects  the core $K_{k-1}$ of $N_{k-1}$ transversally at one point. For $j\in I_{0}\setminus \{j_{0}\}$, there exists a disc $\hat{D}_{j}\subset \partial N_{k-1}$ with boundary $\gamma_{j}$.

Define  a new disc $\hat{D}:=(D\setminus \cup_{j\in I_{0}}D_{j})\cup_{j\in I_{0}}(\cup_{\gamma_{j}}\hat{D}_{j})$ with boundary $\gamma$. It intersects $ K_{k-1}$ transversally at one point, which implies that the intersection number between $\hat{D}$ and $K_{k-1}$ is $\pm 1$.\par
Therefore, the linking number of $\gamma\amalg K_{k-1}$ is $\pm 1$. This is in contradiction with the fact that its linking number is zero. 

This completes the proof of the Claim.

From the Claim, there are at least two distinct meridians, $\gamma_{j_{0}}$ and $\gamma_{j_{1}}$, of $N_{k-1}$. Applying our inductive assumption to $D_{j_{0}}$ and $D_{j_{1}}$ respectively, we have that $D_{j_{t}}\cap\text{Int}~ N_{l}$ has at least $2^{k-1-l}$ components intersecting $N_{0}$ for $t=0,1$ for $l\leq k-1$. Therefore, $D\cap \text{Int}~N_{l}$ has at least $2^{k-l}$ components intersecting $N_{0}$.  
\end{proof}

Based on Theorem \ref{component}, we introduce a topological property.\par

\begin{definition}\label{P} Any contractible genus one 3-manifold $M$ is called to satisfy \emph{Property $P$} if for any properly embedded plane $\Sigma\subset M$, any $k>0$ and any closed curve $\gamma\subset \partial N_{k}\cap \Sigma$, it holds one of the following:
\begin{enumerate}[leftmargin=15pt]
\item $\gamma$ is contractible in $\partial N_{k}$;
\item for $l<k$, $D\cap \text{Int}~N_{l}$ has at least $I(N_{l}, N_{k})$ components intersecting $N_{0}$, \end{enumerate}
where $D\subset \Sigma$ is a unique disc  with boundary $\gamma$ and $\{N_{k}\}_{k}$ is a sequence  as described in Theorem \ref{genus one}.
\end{definition}
We will show that all contractible genus one 3-manifolds satisfy Property $P$ (Theorem \ref{PT}).

\section{Property $P$ and Genus one $3$-manifolds}

In this section, we show that any contractible genus one $3$-manifold satisfies Property $P$.
First, recall some notations from Section 2. Any contractible genus one $3$-manifold $M$ is the ascending union of closed solid tori $\{N_{k}\}_{k=0}^{\infty}$ so that $N_{k}$ is contractible in $N_{k+1}$ and the geometric index $I(N_{k}, N_{k+1})\geq 2$.

In the genus one case,  Lemma \ref{meridian} can be generalized as follows: 

\begin{lemma}\label{meridian2} A circle $\gamma\subset \Sigma\cap \partial N_{k}$, which is not contractible in $\partial N_{k}$, is a meridian of $N_{k}$, where $\Sigma\subset M$ is a properly embedded plane. Moreover, the unique disc $D\subset \Sigma$ with boundary $\gamma$ intersects the core $K_{0}$ of $N_{0}$.
\end{lemma}
\begin{proof} We may assume that $\Sigma$ intersects $\partial N_{k}$ transversally. Since $\Sigma$ is properly embedded, $\Sigma\cap \partial N_{k}:=\{\gamma_{i}\}_{i=0}^{n}$ has finitely many components, where $\gamma_{0}=\gamma$. Each $\gamma_{i}$ bounds a unique closed disc $D_{i}\subset \Sigma$ (where $D_{0}=D$).

Define the set $C:=\{\gamma_{i}| \gamma_{i}\subset D_{0}$ is not contractible in $\partial N_{k}\}$. It is not empty ($\gamma_{0}\subset D_{0}$). 
The disjointness of $\{\gamma_{i}\}_{i=0}^{n}$ implies that  the intersection number of $\gamma\amalg\gamma_{i}$ in $\partial N_{k}$ is zero for each $i\neq 0$. 

If $[\gamma_{i}]$ is not equal to $\pm[\gamma]$ in $\pi_{1}(\partial N_{k})$ for some $\gamma_{i}\in C$, the intersection number of $\gamma$ and $\gamma_{i}$ in $\partial N_{k}$ is nonzero. This contradicts the above fact. We can conclude that each $\gamma_{i}\in C$ is homotopic to $\gamma$ in $\partial N_{k}$. 

\vspace{2mm}

In the following, we will show that each minimal circle $\gamma_{j}$ in $C$ is a meridian. This is to say, $\gamma$ is also a meridian of $N_{k}$. 

The remaining proof is the same as  the proof of  Lemma \ref{one}. It is sufficient to show that $\gamma_{j}$ is homotopically trivial in $N_{k}$. We begin by constructing an immersed disc $\hat{D}_{j}\subset {N}_{k}$ with boundary $\gamma_{j}$. 

\vspace{2mm}

Let us consider the set $C_{j}:=\{\gamma_{i}~| \gamma_{i}\subset \text{Int} ~D_{j}\}\subset C$ and  the set $C^{max}_{j}$ of maximal circles  in $C_{j}$. One has two cases: $C_{j}=\emptyset$ or $C_{j}\neq\emptyset$.

\noindent\textbf{Case I:} If $C_{j}$ is empty, we consider the set $Z:=\text{Int}~D_{j}$ and the disc $\hat{D}_{j}:=D_{j}$;

\noindent\textbf{Case II:} If $C_{j}$ is not empty, then $C^{max}_{j}$ is also non-empty. From the minimality of $\gamma_{j}$, each $\gamma_{i}\in C^{max}_{j}$ is contractible in $\partial N_{k}$ and bounds a disc ${D}'_{i}\subset \partial N_{k}$.

Define the set $Z:=\text{Int}~D_{j}\setminus \cup_{\gamma_{i}\in C^{max}_{j}}D_{i}$ and the disc $\hat{D}_{j}:=Z\cup(\cup_{\gamma_{i}\in C^{max}_{j}}{D}'_{i})$ with boundary $\gamma_{j}$. 

\vspace{1.5mm}

Let us explain why $\hat{D}_{j}$ is contained in $N_k$. In any case, $\partial N_{k}$ cuts $M$ into two connected components, $M\setminus N_{k}$ and $\text{Int}~N_{k}$. The set $Z$ is one of these components of $\text{Int} ~D_{j}\setminus \partial N_{k}$. Therefore, it is in $M\setminus N_{k}$ or $\text{Int}~N_{k}$. 

 If $Z$ is in $M\setminus N_{k}$, then the disc $\hat{D}_{j}$ with boundary $\gamma_{j}$ is contained in $\overline{M\setminus N_{k}}$. Therefore, we see that $[\gamma_{j}]=1$ in $\pi_{1}(\overline{M\setminus N_{k}})$. However, the map $\pi_{1}(\partial N_{k})\rightarrow \pi_{1}(\overline{M\setminus N_{k}})$ is injective (Lemma \ref{injective}). That is to say, $\gamma_{j}$ is null-homotopic in $\partial N_{k}$. This contradicts the fact that $[\gamma_{j}]\neq 1$ in $\pi_{1}(\partial N_{k})$.
 We can conclude that $Z$ is contained in $\text{Int}~N_{k}$.

 \vspace{2mm}
 
Therefore, $\hat{D}_{j}$ is contained in ${N}_{k}$. Its boundary $\gamma_{j}$ is nullhomotopic in ${N}_{k}$. Since $\gamma$ is homotopic to $\gamma_{j}$ in $\partial N_{k}$, it is also contractible in ${N}_{k}$. By Definition \ref{mer}, $\gamma$ must be a meridian of $N_{k}$. 

\vspace{2mm}

By Lemma \ref{injective}, the two induced maps $\pi_{1}(\partial N_{k})\rightarrow \pi_{1}(\overline{M\setminus N_{k}})$ and $\pi_{1}(\partial N_{k})\rightarrow \pi_{1}({N}_{k}\setminus K_{0})$ are both injective. Van-Kampen's theorem shows that $\pi_{1}(M\setminus K_{0})\cong \pi_{1}(\overline{M\setminus N_{k}})\ast_{\pi_{1}(\partial N_{k})}\pi_{1}({N}_{k}\setminus K_{0})$. Further, we have that  $\pi_{1}(\partial N_{k})\rightarrow \pi_{1}(M\setminus K_{0})$ is also injective (see [Theorem 11.67, Page 404] of \cite{R}). Thus, $[\gamma]\neq 1$ in $\pi_{1}(M\setminus K_{0})$. Namely $D\subset \Sigma$ with boundary $\gamma$ must intersect the core $K_{0}$ of $N_{0}$.
\end{proof}
\begin{rem}~~~~~~~~~~
\begin{itemize}[leftmargin=15pt]\item In the proof, the set $\hat{D}_{j}\cap \text{Int}~N_{k}$ is equal to the set $Z$, a subset of $D\cap \text{Int}~N_{k}$. 
\item The disc $\hat{D}_{j}$ may not be embedded, because $D'_{i}$ may be contained in some $D'_{i'}$. When it is not an embedding, we can deform $\hat{D}_{j}$ in a small neighborhood of $\partial N_{k}$ in ${N}_{k}$ so that it becomes an embedded disc in ${N}_{k}$.
\end{itemize}
\end{rem}

\begin{theorem} \label{PT}Any contractible genus one $3$-manifold $M$ satisfies Property $P$.
\end{theorem}
\begin{proof} 
Consider a properly embedded plane $\Sigma\subset M$. Suppose there is some closed curve $\gamma\subset \Sigma\cap \partial N_{k}$ which is not contractible in $\partial N_{k}$ for some $k\in \mathbb{N}_{>0}$.  Lemma \ref{meridian2} shows that $\gamma$ is a meridian of $N_{k}$ and bounds a unique closed disc $D\subset \Sigma$. 

We may assume that $\Sigma$ intersects  $\partial N_{k}$ transversally. 
The set $D\cap \partial N_{k}:=\{\gamma_{i}\}_{i=0}^{n}$ has finitely many components where $\gamma_{0}=\gamma$. 

Define the set $C:=\{ \gamma_{i}~|$ the circle $\gamma_{i}\subset D\cap \partial N_{k}$ is not contractible in $\partial N_{k}\}$.  (It is not empty because $\gamma\subset D$). By Lemma \ref{meridian2},  each minimal circle $\gamma_{j}$ in $C$ is a meridian of $N_{k}$. It bounds a unique closed disc $D_{j}\subset D$. As in the proof of Lemma \ref{meridian2},  we construct a disc $\hat{D}_{j}\subset {N}_{k}$ with boundary $\gamma_{j}$.  Remark that  $\hat{D}_{j}\cap \text{Int}~N_{k}$ is a subset of $D\cap \text{Int}~N_{k}$.

As described in the above remark, the disc $\hat{D}_{j}$ may not be  embedded. If necessary,  we can deform it in a small neighborhood of $\partial N_{k}$ in ${N}_{k}$ so that it becomes an embedded disc.  For $l<k$, $ \hat{D}_{j}\cap \text{Int}~N_{l}$ is still a subset of $D\cap \text{Int}~N_{l} \subset \Sigma$. 

\vspace{2mm}

It is sufficient for us to show that $\hat{D}_{j}\cap \text{Int}~N_{l}$ has at least $I(N_{l}, N_{k})$ components intersecting $N_{0}$. 

 We may assume that $\hat{D}_{j}$ intersects $\partial N_{l}$ transversally. The intersection $\hat{D}_{j}\cap \partial N_{l}:=\{\gamma'_{t}\}_{t\in T}$ has finitely many components. Let us consider the set $\hat{C}^{max}$ of maximal circles in $\{\gamma'_{t}\}_{t\in T}$ and its subset $\hat{C}^{non}:=\{\gamma'_{t}\in \hat{C}^{max}|\gamma'_{t}$ is not contractible in $\partial N_{l}\}$. 
 
 \textbf{Claim:} $|\hat{C}^{non}|\geq I(N_{l}, N_{k})$.
 
We argue by contradiction. Suppose that $|\hat{C}^{non}|<I(N_{l}, N_{k})$. Each $\gamma'_{t}\in \hat{C}^{max}$ bounds a unique disc $D'_{t}\subset \hat{D}_{j}$ .\par

If $\gamma'_{t}$ is in $\hat{C}^{non}$, then it is a meridian of $N_{l}$ (see Lemma \ref{meridian2}). Therefore, we can find a meridian disc $D''_{t}$ of $N_{l}$ which intersects the core $K_{l}$ of $N_{l}$ transversally at one point. If $\gamma'_{t}\in \hat{C}^{max}\setminus \hat{C}^{non}$, $\gamma_{t}'$ is contractible in $\partial N_{l}$ and bounds a disc $D''_{t}$ in $\partial N_{l}$. 

Define a  disc $\hat{D}'_{j}$ with boundary $\gamma_{j}$ \par
\vspace{1mm}
              $\quad\quad\quad\quad\quad\quad\quad\hat{D}'_{j}:=(\hat{D}_{j}\setminus\cup_{\gamma'_{t}\in \hat{C}^{max}} D'_{t}) \cup(\cup_{\gamma'_{t}\in \hat{C}^{max}}{D}''_{t}).$\par
\vspace{1mm}
\noindent The number $\#(\hat{D}'_{j}\cap K_{l})$ of points of $\hat{D}'_{j}\cap K_{l}$ is equal to $|\hat{C}^{non}|<I(N_{l}, N_{k})$. 

 As above, the disc $\hat{D}'_{j}$ may not be embedded (because $D''_{t}$ may be contained in some $D''_{t'}$). If necessary, we  modify the disc $\hat{D}'_{j}$ in a small neighborhood of $\partial N_{l}$ so that it becomes an embedded disc in $N_{k}$. 
 
 Therefore, we may assume that $(\hat{D}'_{j}, \partial \hat{D}'_{j})\subset ({N}_{k}, \partial N_{k})$ is an embedded disc with boundary $\gamma_{j}$. Note that $\gamma_j$ is a meridian of $N_{k}$. Then, $\hat{D}'_{j}$ is a meridian disc of $N_{k}$ with  
$\#(\hat{D}'_{j}\cap K_{l})<I(N_{l}, N_{k})$. 
However, the definition of the geometric index (see Definition \ref{index}) gives that $\#(\hat{D}'_{j}\cap K_{l})\geq I(N_{l}, N_{k})$, a contradiction. This proves the claim. 

\vspace{2mm}

In the following, we will finish the proof of the theorem.

Let $\{\gamma'_{s}\}_{s=1}^{m}$ be the circles in $\hat{C}^{non}$ and ${D}'_{s}\subset \hat{D}_{j}$ the unique disc with boundary $\gamma'_{s}$, where $m=|\hat{C}^{non}|$. From the maximality of $\gamma'_{s}$ in $\{\gamma'_{t}\}_{t\in T}$, $\{{D}'_{s}\}_{s=1}^{m}$ is a family of pairwise disjoint discs in $\hat{D}_{j}$. 

We use  Lemma \ref{meridian2} to see that each $\gamma'_{s}\in \hat{C}^{non}$ is a meridian. In addition, ${D}'_{s}$ intersects the core $K_{0}$ of $N_{0}$. The intersection ${D}'_{s}\cap \text{Int}~N_{l}$ contains at least one component intersecting $N_{0}$. 

We conclude that $\hat{D}_{j}\cap \text{Int}~N_{l}$ has at least $m$ components intersecting $N_{0}$. From the above claim, we know that $m\geq I(N_{l}, N_{k})$. Therefore, $D\cap \text{Int}~N_{l}$ has at least $I(N_{l}, N_{k})$ components intersecting $N_{0}$.
\end{proof}

\section{Stable minimal surfaces}
In this section, we consider a complete Riemannian manifold $(M, g)$ where $M:=\cup_{k=0}^{\infty} N_{k}$ is a contractible genus one $3$-manifold and $\{N_{k}\}_{k}$ is described in Theorem \ref{genus one}. Our first target is to construct a family of meridian discs with ``good" properties. Then, we consider their limit in $(M, g)$. Finally, we investigate the properness of limit surfaces.

\subsection{Stable minimal discs}

Our first object is to construct a family of meridian discs  in $(M, g)$. A result of Meeks and Yau (see  [Theorem 6.28 Page 224] in \cite{CM1})  provides us a geometric version of Loop Theorem to construct them.\par  

\begin{theorem}\label{stable}\textnormal{(See [Theorem 6.28, Page 224] in \cite{CM1})} Let $(M^{3}, g)$ be a compact Riemannian 3-manifold whose boundary is mean convex and $\gamma$ a simple closed curve in $\partial M$ which is null-homotopic in $M$. Then, $\gamma$ bounds an area-minimizing disc and any such least area disc  is embedded.
\end{theorem}

Our strategy is to apply this theorem to $(N_{k}, g|_{N_{k}})$ for any $k>0$. However, its boundary may not be mean convex. To overcome this, we find a new metric $g_{k}$ on $N_{k}$ satisfying 1) $(N_{k}, g_{k})$ is a $3$-manifold with mean convex boundary; 2) $g_{k}|_{N_{k-1}}=g|_{N_{k-1}}$. The metric $g_{k}$ is constructed as below:\par
\vspace{2mm}
\emph{Let $h(t)$ be a positive smooth function on $\mathbb{R}$ so that  $h|_{\mathbb{R}\setminus [-\epsilon, \epsilon]}=1$. Consider the function $f(x):=h(d(x, \partial N_{k}))$ and the metric $g_{k}:=f^{2}g|_{N_{k}}$. Under $(N_{k}, g_{k})$, the mean curvature $\hat{H}(x)$ of $\partial N_{k}$ is $$\hat{H}(x)=h^{-1}(0)(H(x)+2h'(0)h^{-1}(0))$$  
Choosing $\epsilon$ small enough and a function $h$ with $h(0)=2$ and $h'(0)>2\max_{x\in\partial N_{k}} |H(x)|+2$, one gets the metric $g_{k}$ which is the required candidate in the assertion.}

\vspace{2mm}
We use  Theorem \ref{stable} to find an embedded area-minimizing disc $\Omega_{k}$ in $({N}_{k}, g_{k})$, where its boundary $\gamma_{k}$ is a meridian of $N_{k}$. Since the geometric index $I(N_{0}, N_{k})>0$, the disc $\Omega_{k}$ intersects $N_{0}$.

To sum up, we construct a family of meridian discs $\{\Omega_{k}\}_{k\in \mathbb{N}}$ intersecting $N_{0}$. Since we modify the metric near $\partial N_{k}$, each $\Omega_{k}$ is not minimal in $(M, g)$, but stable minimal away from the neighborhood of $\partial N_{k}$ (near $N_{k-1}$, for example).

The surface $\Omega_{k}\cap N_{k-1}$ may be disconnected. However, it is a stable minimal lamination (see Definition \ref{lam}), where each leaf has boundary contained in $\partial N_{k-1}$. 
 \subsection{Limit surfaces}
\begin{definition}\label{converge} In a complete Riemannian $3$-manifold $(M, g)$, a sequence $\{\Sigma_{n}\}$ of immersed minimal surfaces \emph{converges smoothly with finite multiplicity} (at most $m$) to an immersed minimal surface $\Sigma$, if for each point $p$ of $\Sigma$, there is a disc neighborhood $D$ in $\Sigma$ of $p$, an integer $m$ and  a neighborhood $U$ of $D$ in $M$ (consisting of geodesics of $M$ orthogonal to $D$ and centered at the points of $D$) so that for $n$ large enough, each $\Sigma_{n}$ intersects $U$ in at most $m$ connected components. Each component is a graph over $D$ in the geodesic coordinates.  Moreover, each component converges to $D$ in $C^{2, \alpha}$-topology as $n$ goes to infinity. 

Note that in the case that each $\Sigma_{n}$ is embedded, the surface $\Sigma$ is also embedded. The \emph{multiplicity} at $p$ is equal to the number of connected component of $\Sigma_{n}\cap U$ for $n$ large enough. It remains constant on each component of $\Sigma$.  

\end{definition}

\begin{remark}\label{graph} Let us consider a family $\{\Sigma_{n}\}_{n}$ of properly embedded minimal surfaces  converging to the minimal surface $\Sigma$ with finite multiplicity. 
Fix a compact simply-connected subset $D\subset \Sigma$. Let $U$ be the tubular neighborhood of $D$ in $M$ with radius $\epsilon$ and  $\pi: U \rightarrow D$ the projection from $U $ onto $D$. It follows that the restriction $\pi|_{\Sigma_{n}\cap U}: \Sigma_{n}\cap U\rightarrow D$ is a $m$-sheeted covering map for $\epsilon$ small enough and $n$ large enough, where $m$ is the multiplicity. 

Therefore, the restriction of $\pi$ to each component of $\Sigma_{n}\cap U$ is also a covering map. Hence, since $D$ is simply-connected, it is bijective. Therefore, each component of $\Sigma_{n}\cap U$ is a normal graph over $D$.
\end{remark}
Let us recall a classical theorem about convergence with finite multiplicity.

\begin{theorem}\label{convergence}\textnormal{(See [Compactness Theorem, Page 96] of \cite{And} and  [Theorem 4.37, Page 49] of \cite{MRR})}
Let $\{\Sigma_{k}\}_{k\in \mathbb{N}}$ be a family of properly embedded minimal surfaces  in a 3-manifold $M^{3}$ satisfying (1) each $\Sigma_{k}$ intersects a given compact set $K_{0}$; (2) for any compact set $K$ in $M$, there are three constants $C_{1}=C_{1}(K)>0$, $C_{2}=C_{2}(K)>0$ and $j_{0}=j_{0}(K)\in \mathbb{N}$ such that for each $k\geq j_{0}$, it holds that
\begin{enumerate}
\item $|A_{\Sigma_{k}}|^{2}\leq C_{1}$ on $K\cap\Sigma_{k}$, where $|A_{\Sigma_{k}}|^{2}$ is the squared length of the second fundamental form of $\Sigma_{k}$
\item $\text{Area}(\Sigma_{k}\cap K)\leq C_{2}$.
\end{enumerate}
Then, after passing to a subsequence, $\Sigma_{k}$ converges to a properly embedded minimal surface with finite multiplicity in the $C^{\infty}$-topology. 
\end{theorem}
Note that the limit surface may be disconnected. \par

Let us consider the sequence $\{\Omega_{k}\}_{k\in \mathbb{N}}$ (constructed in Section 5.1). 
The sequence $\{\Omega_{k}\}_{k\in \mathbb{N}}$ satisfies Condition (1) in Theorem \ref{convergence} (see  [Theorem 3, Page 122] in \cite{Sc}). Say, they have locally uniformly-bounded curvatures.
 However, this sequence may not satisfy Condition (2) of Theorem \ref{convergence}.
 
For example, in the Whitehead manifold,  for $k>1$, $\Omega_{k}\cap \text{Int}~N_{1}$ has at least $2^{k-1}$ components intersecting $N_{0}$ (see Theorem \ref{component}). We know that each component $(\Sigma, \partial \Sigma)\subset (N_{1}, \partial N_{1})$ of $\Omega_{k}\cap \text{Int}~N_{1}$ intersecting $N_{0}$ is a minimal surface in $(M, g)$.

 Choose $x_{0}\in \Sigma\cap N_{0}$ and $r_{0}=\frac{1}{2}\min \{r, i_{0}\}$, where $r:=\text{dist}(\partial N_{0}, \partial N_{1})$ and $i_{0}:=\inf_{x\in N_{1}}\text{Inj}_{M}(x)$. We see that the ball $B(x_{0}, r_{0})$ is in $N_{1}$. 
 We apply a result (see [Lemma 1, Page 445] in \cite{YM}) to $(N_{1}, \partial N_{1})$. Hence, it follows that  
$$\text{Area}(\Sigma)\geq \text{Area}(\Sigma\cap B(x_{0}, r_{0}))\geq C(i_{0}, r_{0}, K)$$
where $K:=\sup_{x\in N_{1}} |K_{M}(x)|$ and $K_{M}$ is the sectional curvature.

Therefore, one has that $\text{Area}(N_{1}\cap \Omega_{k})\geq 2^{k-1}C$.  That is to say, the sequence $\{\Omega_{k}\}_{k}$  does not satisfy Condition (2). 

Generally, the sequence $\{\Omega_{k}\}_{k}$ may not sub-converge with finite multiplicity. 

In the following, we consider the convergence towards a lamination. 

\begin{definition}\label{lam}A \emph{codimension one lamination} in a 3-manifold $M^{3}$ is a collection $\mathscr{L}$ of smooth disjoint surfaces (called leaves) such that $\bigcup_{L \in \mathscr{L}}L$ is closed in $M^{3}$. Moreover, for each $x\in M$ there exists an open neighborhood $U$ of $x$ and a coordinate chart $(U, \Phi)$, with $\Phi(U)\subset \mathbb{R}^{3}$ so that in these coordinates the leaves in $\mathscr{L}$ pass through $\Phi(U)$ in slices form $$\mathbb{R}^{2}\times \{t\}\cap\Phi(U).$$A \emph{minimal lamination} is a lamination whose leaves are minimal. Finally, a sequence of laminations is said to converge if the corresponding coordinate maps converge. 
\end{definition}

For example, $\mathbb{R}^{2}\times\Lambda$ is a lamination in $\mathbb{R}^{3}$, where $\Lambda$ is a closed set in $\mathbb{R}$. 

Note that any (compact) embedded surface (connected or not) is a lamination. 
In [Appendix B, Laminations] of \cite{CM}, Coding-Minicozzi described the limit of laminations with uniformly bounded curvatures. 

\begin{proposition}\label{Lam}\textnormal{(See [Proposition B.1, Page 610] in \cite{CM})} Let $(M^{3}, g)$ be a fixed 3-manifold. If $\mathscr{L}_{i}\subset B(x, 2R)\subset M$ is a sequence of minimal laminations with uniformly bounded curvatures (where each leaf has boundary contained in $\partial B(0, 2R)$), then a subsequence, $\mathscr{L}_{j}$, converges in the $C^{\alpha}$-topology for any $\alpha<1$ to a  lamination $\mathscr{L}$ in $B(x, R)$ with minimal leaves. 
\end{proposition}

\begin{exam} Let $(\mathbb{R}^3, g_{Eucl})$ be the $3$-dimensional Euclidean space and $\mathscr{L}_n:=\mathbb{R}^2\times \{\frac{i}{n}\}^{n}_{i=1}\cap B(0, 2)$, where $B(0, 2):=\{x\in \mathbb{R}^3 |~|x|\leq 2\}$.

Each set $\mathscr{L}_n$ is a lamination whose each leaf is stable minimal . Furthermore, the curvature vanishes on each leaf of $\mathscr{L}_n$. Since $\lim_{n\rightarrow \infty}\text{Area} (\mathscr{L}_n)=\infty$, it follows that the sequence $\{\mathscr{L}_n\}_n$ does not converge with finite multiplicity.  However,  the sequence $\{\mathscr{L}_n\}$ converges to the lamination, $\mathscr{L}:=\mathbb{R}^2\times [0,1] \cap B(0,2)$. 
\end{exam}

The convergence of the laminations is equivalent  to the convergence of the corresponding coordinate maps. Thus, the coordinate maps of $\mathscr{L}$ (in Proposition \ref{Lam}) are both in $C^{\alpha}$ (since $\{\mathscr{L}_j\}$ sub-converges in $C^{\alpha}$). Namely, each leaf is a $C^{\alpha}$-surface. (Such a lamination is called a Lipschitz lamination). In addition, each leaf of $\mathscr{L}$ satisfies minimal hypersurface equation. The regularity theory of elliptic equations (see \cite{GT}) implies that each leaf is smooth ($C^\infty$).

 We apply this proposition to the sequence $\{\Omega_{k}\}_{k}$ and have that

\begin{theorem}\label{limits}The sequence $\{\Omega_{k}\}_{k\in \mathbb{N}}$ sub-converges to a lamination $\mathscr{L}=\amalg_{t\in \Lambda}L_{t}$. Moreover, each $L_{t}$ is a complete minimal surface.
\end{theorem} 

\begin{proof}  
For any $j>k$, $\Omega_{j}\cap N_{k}$ is  disconnected but is a stable minimal lamination. Each leaf has boundary contained in $\partial N_{k}$. 

A result of Schoen (see [Theorem 3, Page 122] in \cite{Sc})  gives a constant $C(N_{k-1})$, depending on $N_{k}$ and $g$, so that for any $j>k$, $$|A_{\Omega_{j}}|^2\leq C(N_{k-1})~ \text{on} ~N_{k-1}$$

Therefore,  $\{\Omega_{j}\cap N_{2}\}^\infty_{j=4}$ is a sequence of minimal laminations with uniformly bounded curvature where each leaf has boundary contained in $\partial N_{2}$. We use Proposition \ref{Lam} to extract a subsequence converging to a minimal lamination in $N_{1}$. The boundary of the lamination is contained in $\partial N_{1}$. 

We repeat the argument on each $N_{k}$. A diagonal argument allows us  to find a subsequence of $\{\Omega_{k}\}$ converging to a lamination $\mathscr{L}$. Each leaf is a complete minimal surface.     
\end{proof}

In addition, we have
 \begin{theorem}\label{stability} Each leaf in $\mathscr{L}$ is stable minimal. 
 \end{theorem}
 
 \begin{theorem}\label{non-compact} If $(M, g)$ has positive scalar curvature, each leaf in $\mathscr{L}$ is non-compact.  
 \end{theorem}
 
We will prove these two theorems in Appendix B. 

\vspace{2mm}

To sum up, the sequence $\{\Omega_{k}\}$ sub-converges to a minimal lamination $\mathscr{L}$. If $(M, g)$ has positive scalar curvature, each leaf in $\mathscr{L}$ is a complete (non-compact) stable minimal surface. 
 
\subsection{Properness}\label{3.2}Theorem \ref{limits} and Theorem \ref{stability} imply that there are some complete embedded stable minimal surfaces in $(M, g)$. It is natural to ask whether such minimal surfaces are proper. In this part, we show that in any complete 3-manifold with positive scalar curvature,  the integral of the scalar curvature along any complete (non-compact) embedded stable minimal surface is bounded by a universal constant, which implies that such a surface is proper. 

\begin{theorem} \label{proper}
Let $(M, g)$ be a complete oriented 3-manifold with positive scalar curvature $\kappa(x)$. Assume that $\Sigma$ is a complete (non-compact) stable minimal surface in $M$. Then, one has,$$\int_{\Sigma}\kappa(x)dv\leq 2\pi,$$where $dv$ is the volume form of the induced metric $ds^{2}$ over $\Sigma$. Moreover, if $\Sigma$ is an embedded surface, then $\Sigma$ is properly embedded.
\end{theorem}

\begin{proof} By [Theorem 2, Page 211] in \cite{SY}, $\Sigma$ is conformally diffeomorphic to $\mathbb{R}^{2}$.\par

Let us consider the Jacobi operator $L:=\Delta_{\Sigma}- K_{\Sigma}+(\kappa(x)+\frac{1}{2}|A|^{2})$, where $K_{\Sigma}$ is the Gaussian curvature of the metric $ds^{2}$ and $\Delta_{\Sigma}$ is the Laplace-Beltrami operator of $(\Sigma, ds^{2})$. From [Theorem 1, Page 201] in \cite{FS}, there exists a positive function $u$ on $\Sigma$ satisfying $L(u)=0$, since $\Sigma$ is a stable minimal surface. \par

Consider the metric $d\tilde{s}^{2}:=u^{2}ds^{2}$. Let $\tilde{K}_{\Sigma}$ be its sectional curvature and $d\tilde{v}$ its volume form. We have that 
$$\tilde{K}_{\Sigma}=u^{-2}(K_{\Sigma}-\Delta_{\Sigma}~{log}~ u)~~~~\text{and}~~~~~~~~d\tilde{v}=u^{2}dv.$$
By Fischer-Colbrie's work (see [Theorem 1, Page 126] in \cite{F}), $(\Sigma, d\tilde{s}^{2})$ is a complete surface with non-negative sectional curvature $\tilde{K}_{\Sigma} \geq 0$.  Cohn-Vossen's inequality \cite{CV} shows that 

 $$\int_{\Sigma}\tilde{K}_{\Sigma}d\tilde{v}\leq 2\pi\chi(\Sigma),$$where $\chi(\Sigma)$ is the Euler characteristic of $\Sigma$.\par
 Since $L(u)=0$, one has that $\int_{B^{\Sigma}(0, R)} L(u)u^{-1}dv=0$, where $B^{\Sigma}(0,R)$ is the geodesic ball in $(\Sigma, ds^{2})$ centered at $0\in \Sigma$ with radius $R$. We deduce that
\begin{equation*}
	\begin{split}
	\int_{B^{\Sigma}(0, R)}\kappa(x)+\frac{1}{2}|A|^{2}dv &=\int_{B^{\Sigma}(0, R)}(K_{\Sigma}-u^{-1}\Delta_{\Sigma}~u)dv\\
	&=\int_{B^{\Sigma}(0, R)} K_{\Sigma}-(\Delta_{\Sigma}~log~u+u^{-2}|\nabla u|)dv\\
	&\leq \int_{B^{\Sigma}(0, R)}u^{-2}(K_{\Sigma}-\Delta_{\Sigma} log~u)u^{2}dv\\
	&=\int_{B^{\Sigma}(0, R))}\tilde{K}_{\Sigma}d\tilde{v}\\
	&\leq \int_{\Sigma} \tilde{K}_{\Sigma}d\tilde{v}
	\end{split}
	\end{equation*}
Since $\Sigma$ is diffeomorphic to $\mathbb{R}^{2}$, then $\chi(\Sigma)=1$. Combining these two inequalities above and taking $R\rightarrow \infty$, we have that, 
$$\int_{\Sigma}\kappa(x)+\frac{1}{2}|A|^{2} dv\leq 2\pi.$$	

\vspace{2mm}

Suppose that $\Sigma$ is not properly embedded.  There is an accumulation point $p$ of $\Sigma$ so that the set $B(p, r/2)\cap \Sigma$ is a non-compact closed set in $\Sigma$. Namely, it is unbounded in $(\Sigma, ds^2)$.
Hence, there is a sequence $\{p_{k}\}$ of points in $\Sigma\cap B(p, r/2)$ going to infinity in $(\Sigma, ds^2)$. 

\vspace{1mm}

We may assume that the geodesic discs $\{B^{\Sigma}(p_k, r/2)\}_k$ in $\Sigma$ are disjoint.

Define two constants $r_{0}:=\frac{1}{2}\min\{r, i_{0}\}$ and $K:=\sup_{x\in B(p, r)}|K_{M}(x)|$ where $i_{0}:=\inf_{x\in B(p, r)} \text{Inj}_{M}(x)$ and $K_{M}$ is the sectional curvature. The geodesic ball $B^{\Sigma}(p_{k}, r_{0}/2)$ is in $B(p, r)$. 

A result (see [Theorem 3, Appendix, Page 139] of \cite{FKR})  shows that
  $$\text{Area}(B^{\Sigma}(p_{k}, r_{0}/2))\geq C(i_{0}, r_{0}, K).$$ 
 This leads to a contradiction as follows:
\begin{equation*}
\begin{split}
2\pi &\geq \int_{\Sigma}\kappa(x)dv \geq \int_{B(p, r)\cap \Sigma} \kappa(x)dv\\
      &\geq \sum_{k} \int_{B^{\Sigma}(p_{k},r_{0}/2)} \kappa(x)\\
      &\geq \inf_{x\in B(p, r)} \kappa(x)\cdot\sum_{k} \text{Area}(B^{\Sigma}(p_{k}, r_{0}/2))\\ 
      &\geq \inf_{x\in B(p, r)} \kappa(x)\cdot\sum_{k} C= \infty
      \end{split}
\end{equation*}
\end{proof}

\begin{remark} From Theorem \ref{limits} and Theorem \ref{stability}, $\{\Omega_{k}\}_{k}$ sub-converges to a lamination $\mathscr{L}=\amalg_{t\in \Lambda}L_{t}$, where each $L_{t}$ is a complete embedded  minimal surface. If $(M, g)$ has positive scalar curvature, we use Theorem \ref{non-compact} and  Theorem \ref{proper} to have that each $L_{t}$ is stable minimal and properly embedded.
\end{remark}

\begin{corollary} Let (M, g) be a contractible complete 3-manifold with uniformly positive scalar curvature ($\inf_{x\in M}\kappa(x)>0$). Then $M$ is homeomorphic to $\mathbb{R}^{3}$
\end{corollary}

\begin{proof} Suppose that $M$ is not simply connected at infinity. As described in Section 4.1 and Theorem \ref{limits}, there exists a complete (non-compact) stable minimal surface $\Sigma$. By Schoen-Yau's Theorem [Theorem 2, Page 211] in \cite{SY}, $\Sigma$ is conformally diffeomorphic to $\mathbb{R}^{2}$.\par
 Since the scalar curvature $\kappa(x)$ of $M$ is uniformly positive, $\inf_{x\in M}\kappa(x)>0$. From Theorem \ref{proper}, one has, 
 \begin{equation*}
 \begin{split}
 2\pi &\geq \int_{\Sigma} \kappa(x)dv\\
        &\geq \inf_{x\in M}\kappa(x)\cdot \int_{\Sigma}dv\\
        &=\inf_{x\in M}\kappa(x) \cdot Area(\Sigma). 
 \end{split}
 \end{equation*}
Therefore, $\Sigma$ is a surface of finite area. \par
 However, we apply the theorem of Gromov and Lawson (see Theorem 8.8 of  \cite{GL}). This theorem asserts that if $(X, g)$ is a Riemannian manifold of positive scalar curvature, then any complete stable minimal surface of finite area in $X$ is homeomorphic to $\mathbb{S}^{2}$. Hence, $\Sigma$ is homeomorphic to $\mathbb{S}^{2}$, which leads to a contradiction with the topological structure of $\Sigma$ ($\Sigma$ is homeomorphic to $\mathbb{R}^{2}$).
\end{proof}

\section{Topological property of Minimal Laminations}
In the rest, we will use the following notations: 

\begin{itemize}[leftmargin=15pt]
\item we let $(M, g)$ denote a complete contractible genus one $3$-manifold; 
\item the manifold $M$ is an increasing union of solid tori $\{N_k\}_{k=1}^\infty$ and the geometric index $I(N_{k}, N_{k+1})\geq 2$ for each $k$ (see Theorem \ref{genus one}); 
\item we define $\mathscr{L}:=\cup_{t\in \Lambda}$ to be a lamination in $(M, g)$ whose each leaf is a complete (non-compact) stable minimal surface.
\end{itemize}

In this section, we study the effect of positive scalar curvature on minimal surfaces and minimal laminations. Precisely,  we use the positivity of scalar curvature to show that: 

\vspace{2mm}

\noindent \textbf{Theorem} \ref{D}  \emph{If $(M, g)$ has positive scalar curvature, then there is an integer $k_0>0$ so that for $k\geq k_0$, $\mathscr{L}\cap \p N_k(\epsilon)$ is conatined in a disjoint union of \textbf{finitely many} discs in $\p N_k(\epsilon)$, where $N_{k}(\epsilon):=N_{k}\setminus N_{\epsilon}(\partial N_{k})$, $N_{\epsilon}(\partial N_{k})$ is some tubular neighborhood of $\partial N_{k}$. }

\subsection{Topological property of each leaf in a minimal lamination} Theorem \ref{PT} shows that any contractible genus one contractible $3$-manifold satisfies Property $P$. We combine  the positivity of scalar curvature to show that 

\vspace{2mm}

\noindent\textbf{Lemma} \ref{trivial}  \emph{Let (M, g) and $\mathscr{L}$ be assumed as above. If $(M, g)$ has positive scalar curvature, then there exists a positive  integer $k_{0}=k_{0}(M, g)$, such that for each $k\geq k_{0}$ and any $t\in \Lambda$, each embedded circle $\gamma$ in $L_{t}\cap \partial N_{k}$ is contractible in $\partial N_{k}$. }

\begin{proof}  The positivity of  scalar curvature tells that each $L_{t}$ is conformally diffeomorphic to $\mathbb{R}^{2}$ (see  [Theorem 2, Page 211] in \cite{SY}). 

We argue by contradiction. Suppose the contrary that there is a sequence of increasing integers $\{k_{n}\}_{n}$ such that :\par
\vspace{2mm}
{\noindent\emph{ for each $k_{n}$, there is a minimal surface $L_{t_{n}}$ in $\{L_{t}\}_{t\in \Lambda}$ and an embedded curve $c_{k_{n}}\subset L_{t_{n}}\cap\partial N_{k_{n}}$ which is not contractible in $\partial N_{k_{n}}$.} }\par
\vspace{2mm}
\noindent Note that since $\lim\limits_{n\rightarrow\infty}k_{n}=\infty$,  it follows that $\lim\limits_{n\rightarrow\infty} I(N_{1}, N_{k_{n}})=\infty$.

Each $c_{k_n}$ bounds a unique disc $D_{n}\subset L_{t_{n}}$ (since $L_{t_n}\cong \mathbb{R}^2$). From Property $P$ (see Section 4),  $D_{n}\cap \text{Int}~N_{1}$ has at least $I(N_{1}, N_{k_{n}})$ components intersecting $N_{0}$, denoted by $\{\Sigma_{j}\}_{j=1}^{m}$.

Define the constants $r:=d^{M}(\partial N_{0},\partial N_{1})$, $C:=\inf_{x\in N_{1}}\kappa(x)$, $K:=\sup_{x\in N_{1}}|K_{M}|$ and $i_{0}:=\inf_{x\in N_{1}}(\text{Inj}_{M}(x))$ , where $K_{M}$ is the sectional curvature of $(M, g)$ and  $\text{Inj}_{M}(x)$ is the injective radius at $x$ of $(M, g)$.\par
 Choose $r_{0}=\frac{1}{2}\min\{i_{0}, r\}$ and $x_{j}\in \Sigma_{j}\cap N_{0}$, then $B(x_{j}, r_{0})$ is in $ N_{1}$. We apply [Lemma 1, Page 445] in \cite{YM} to the minimal surface $(\Sigma_{j}, \partial\Sigma_{j})\subset (N_{1}, \partial N_{1})$. Hence, one has that 
$$\text{Area}(\Sigma_{j}\cap B(x_{j}, r_{0}))\geq C_{1}(K, i_{0}, r_{0}).$$
From Theorem \ref{proper}, we have:
\begin{equation*}
\begin{split}
2\pi &\geq \int_{L_{t_{n}}} \kappa(x) dv\geq \sum_{j=1}^{m}\int_{\Sigma_{j}} \kappa(x) dv \geq \sum_{j=1}^{m}\int_{\Sigma_{j}\cap B(x_{j}, r_{0})} \kappa(x) dv\\
       &\geq \sum_{j=1}^{m} C\text{Area}(\Sigma_{j}\cap B(x_{j}, r_{0})) \\
       &\geq CC_{1}m\geq CC_{1}I(N_{1}, N_{k_{n}})
       \end{split}
\end{equation*}It contradicts  the fact that $\lim\limits_{n\rightarrow\infty}I(N_{1}, N_{k_{n}})=\infty$ and completes the proof.\end{proof}
\begin{remark}\label{trivial1} The proof of Lemma \ref{trivial} only depends on the extrinsic Cohn-Vossen's inequality (see Theorem \ref{proper}) and the geometric index. If $N_{k}$ is replaced by $N_{k}(\epsilon)$, all geometric indexes remain unchanged where $N_{k}(\epsilon):=N_{k}\setminus N_{\epsilon}(\partial N_{k})$ and $N_{\epsilon}(\partial N_{k})$ is a 2-sided tubular neighborhood of $\partial N_{k}$ with small radius $\epsilon$. Therefore, Lemma \ref{trivial} is also valid for $\partial N_{k}(\epsilon)$. 
\end{remark}

We now use Lemma \ref{trivial} to explain Theorem \ref{D} for the case when $\mathscr{L}$ has finitely many components.

We may assume that each leaf $L_{t}$ intersects $\partial N_{k}$ transversally for a fixed constant $k\geq k_{0}$, where $k_0$ is defined in Lemma \ref{trivial}. Since $L_{t}$ is properly embedded (see Theorem \ref{proper}), $L_{t}\cap \partial N_{k}:=\{\gamma^{t}_{i}\}_{i\in I_{t}}$ has finitely many components. Each component is an embedded closed curve.
Lemma \ref{trivial} shows that each $\gamma^{t}_{i}$ is contractible in $\partial N_{k}$ and  bounds a disc $D^t_i\subset \p N_k$. 

Consider  the partially ordered set $(\{D^{t}_{i}\}_{t\in \Lambda, i\in I_{t}}, \subset)$. Let $C$ be the set of maximal elements. It is a \textbf{finite} set. The set $\mathscr{L}\cap \partial N_{k}$ is contained in the disjoint union of  closed discs in $C$. 

\subsection{Setup of the proof for Theorem \ref{D}}
In order to prove Theorem \ref{D}, we will introduce a set $S$ and prove its finiteness which will imply Theorem \ref{D}. We begin with two topological lemmata.

\begin{lemma}\label{existence} Let $(\Omega, \partial \Omega) \subset (N, \partial N)$ be a 2-sided embedded disc with some open sub-discs removed, where $N$ is a closed solid torus. Assume that each circle $\gamma_{i}$ is contractible in $\partial N$, where $\partial \Omega=\amalg_{i}\gamma_{i}$. Then $N\setminus \Omega$ has two connected components. Moreover, there is a unique component $B$ satisfying that the induced map $\pi_{1}(B)\rightarrow\pi_{1}(N)$ is trivial.  
\end{lemma}

\begin{lemma}\label{sep} Let $(\Omega_{1}, \partial\Omega_{1})$ and $(\Omega_{2}, \partial\Omega_{2})$ be two disjoint surfaces as  in Lemma \ref{existence}. Assume that for $t=1,2$, $N\setminus \Omega_{t}$ has a unique component $B_{t}$ with the property that the map $\pi_{1}(B_{t})\rightarrow \pi_{1}(N)$ is trivial. Then one of the following holds:
\begin{enumerate}
\item $B_{1}\cap B_{2}=\emptyset$;
\item $B_{1}\subset B_{2}$;
\item$B_{2}\subset B_{1}$.
\end{enumerate} 
 \end{lemma}
 
 We will show these two lemmata in Appendix A.

 \vspace{2mm}
 
 In the rest, we use these two topological lemmata to construct the set $S$ and show its finiteness. 
 
 \subsubsection{\textbf{The element in $S$}} Lemma \ref{trivial} shows that there is an integer $k_{0}>0$ satisfying that for any $k\geq k_{0}$ and $t\in \Lambda$, each circle  in $L_{t}\cap \partial N_{k}$ is null-homotopic  in $\partial N_{k}$. 
 
 In the following, we will work on the open solid torus $\text{Int} ~N_k$ and construct the set $S$, for a fixed integer $k\geq k_0$. 
 
 \vspace{2mm}
 
 Let $\{\Sigma^{t}_{i}\}_{i\in I_{t}}$ be  the set of components of $L_{t}\cap \text{Int} ~N_{k}$ for each $t\in\Lambda$. (It may be empty.)  We will show that for each component $\Sigma^{t}_{i}$, $N_{k}\setminus \overline{\Sigma^{t}_{i}}$ has a unique component $B^{t}_{i}$ satisfying that $\pi_{1}(B^{t}_{i})\rightarrow \pi_{1}(N_{k})$ is trivial. 

\vspace{2mm}

If $L_{t}$ intersects $\p N_{k}$ transversally,  the boundary $\p \Sigma^{t}_{i}\subset L_{t}\cap \p N_{k}$ is the union of some disjointly embedded circles. We have that 

\begin{itemize}[leftmargin=15pt]
\item each circle in the boundary $\partial \Sigma^{t}_{i}\subset L_{t}\cap \p N_{k}$ is contractible in $\p N_{k}$ (see Lemma \ref{trivial});
\item $\Sigma^{t}_{i}$ is homeomorphic to an open disc with some disjoint closed subdiscs removed (since $L_t\cong \mathbb{R}^2$).
\end{itemize}  By Lemma \ref{existence}, $N_{k}\setminus \overline{\Sigma^{t}_{i}}$ has a unique component $B^{t}_{i}$ satisfying that $\pi_{1}(B^{t}_{i})\rightarrow \pi_{1}(N_{k})$ is trivial.

\vspace{2mm}

In general, $L_{t}$ may not intersect $\p N_{k}$ transversally. To overcome it, we will deform the surface $\p N_{k}$.  Precisely, for the leaf $L_{t}$, we find  a new solid torus $\tilde{N}_{k}(\epsilon_{t})$ so that $L_{t}$ intersects $\p \tilde{N}(\epsilon_{t})$ transversally, where $\tilde{N}_{k}(\epsilon_{t})$ is a closed tubular neighborhood of $ N_{k}$ in $M$.   

Consider the component $\tilde{\Sigma}^{t}_{i}$ of $L_{t}\cap \text{Int} ~\tilde{N}_{k}(\epsilon_{t})$ containing $\Sigma^{t}_{i}$. As above,  $\tilde{N}_{k}~(\epsilon_{t})\setminus \overline{\tilde{\Sigma}^{t}_{i}}$ has a unique component $\tilde{B}^{t}_{i}$ so that the map $\pi_{1}(\tilde{B}^{t}_{i})\rightarrow \pi_{1}(\tilde{N}_{k}(\epsilon_{t}))$ is trivial. 

Choose the component $B^{t}_{i}$ of $\tilde{B}^{t}_{i}\cap N_{k}$ whose boundary contains $\Sigma^{t}_{i}$. It is a component of $N_{k}\setminus \overline{\Sigma^{t}_{i}}$. In addition, the map $\pi_{1}(B^{t}_{i})\rightarrow \pi_{1}(\tilde{B}^{t}_{i})\rightarrow \pi_{1}(\tilde{N}_{k}(\epsilon_{t}))$ is trivial. Since $N_{k}$ and $\tilde{N}_{k}(\epsilon_{t})$ are homotopy equivalent, the map $\pi_{1}(B^{t}_{i})\rightarrow \pi_{1}({N}_{k})$ is also trivial.  This finishes the construction of $B^{t}_{i}$.

\subsubsection{\textbf{Properties of $S$}}
 From Lemma \ref{sep}, for any $B^{t}_{i}$ and $B^{t'}_{i'}$, one of the following holds: 
 
 \begin{enumerate}
\item $B^{t}_{i}\cap B^{t'}_{i'}=\emptyset$;
\item $B^{t}_{i}\subset B^{t'}_{i'}$;
\item$B^{t'}_{i'}\subset B^{t}_{i}$,
\end{enumerate} 
 where $t, t'\in \Lambda$, $i\in I_{t}$ and $i'\in I_{t'}$. 
 
  Therefore, $(\{B^{t}_{i}\}_{t\in \Lambda, i\in I_{t}}, \subset)$ is a partially ordered set. We consider  the set $\{B_{j}\}_{j\in J}$ of maximal elements. It  may be infinite.

\begin{definition}
  $S:=\{B_{j}~:~ B_{j}\cap N_{k}(\epsilon/2)\neq \emptyset\}$ where $N_{k}(\epsilon/2):=N_{k}\setminus N_{\epsilon/2}(\partial N_{k})$ and $N_{\epsilon/2}(\partial N_{k})$ is a 2-sided tubular neighborhood of $\partial N_{k}$ with radius $\epsilon/2$.
\end{definition}

\begin{proposition}\label{intersect} Let $\Sigma^{t}_{i}$ be a component of $L_{t}\cap \text{Int}~N_{k}$ and $B^{t}_{i}$ assumed as above. If $B^{t}_{i}$ is an element in $S$, then $\Sigma^{t}_{i}\cap N_{k}(\epsilon/2)$ is nonempty. \end{proposition}
\begin{proof}
We argue by contradiction. Suppose that  $\Sigma^{t}_{i}\cap N_{k}(\epsilon/2)$ is empty.  As mentioned above, $\overline{\Sigma^{t}_{i}}$ cuts $N_{k}$ into two components. Hence, $N_{k}(\epsilon/2)$ must be in one of these two components. 

The definition of  $S$ shows that the component $B^{t}_{i}$ of $N_{k}\setminus \overline{\Sigma^{t}_{i}}$ must intersect $N_{k}(\epsilon/2)$. Thus, $N_{k}(\epsilon/2)$ is contained in $B^{t}_{i}$. 

However, the composition of maps $\pi_{1}(N_{k}(\epsilon/2))\rightarrow \pi_{1}(B^{t}_{i})\rightarrow \pi_{1}(N_{k})$ is an isomorphism. Namely, the map $\pi_{1}(B^{t}_{i})\rightarrow \pi_{1}(N_{k})$ is non-trivial and surjective, which leads to a contradiction with the fact that the map $\pi_{1}(B^{t}_{i})\rightarrow \pi_{1}(N_{k})$ is trivial.
 This finishes the proof.
\end{proof}

 \begin{proposition}\label{cover} Let $N_k$, $B_j$ and $S$ be assumed as above. Then one has that
 $N_{k}(\epsilon)\cap \mathscr{L}\subset \bigcup_{{B}_{j} \in S} \overline{B}_{j}\cap N_{k}(\epsilon)$. Moreover, $\partial N_{k}(\epsilon)\cap \mathscr{L}\subset \bigcup_{{B}_{j}\in S} \overline{B}_{j}\cap \partial N_{k}(\epsilon)$.
 \end{proposition}
  
\begin{proof} Each component $\Sigma^{t}_{i}$ of $L_{t}\cap \text{Int}~N_{k}$ is contained in $\overline{B^{t}_{i}}$. Hence, $L_{t}\cap \text{Int} ~N_{k}$ is in $\cup_{i\in I_{t}}\overline{B^{t}_{i}}$. We can conclude that $\mathscr{L}\cap\text{Int}~ N_{k}$ is contained in $\cup \overline{B^{t}_{i}} $. 

By the maximality of $\{B_j\}_{j\in J}$, the set $\cup B^{t}_{i}$ is equal to $\cup_{j\in J} B_{j}$. 
Therefore, $\mathscr{L}\cap \text{Int}~N_{k}$ is in $\cup_{j\in J} \overline{B_{j}}$. 
In addition, the definition of $S$ gives that $\cup_{j\in J}\overline{B_{j}}\cap N_{k}(\epsilon)=\cup_{B_{j}\in S}\overline{B_{j}}\cap N_{k}(\epsilon)$. Therefore, $N_{k}(\epsilon)\cap \mathscr{L}\subset \cup_{B_{j}\in S} \overline{B_{j}}\cap N_{k}(\epsilon)$. 

Similarly, one has that $\cup_{j\in J}\overline{B_{j}}\cap \partial N_{k}(\epsilon)$ equals $\cup_{B_{j}\in S}\overline{B_{j}}\cap \partial N_{k}(\epsilon)$. Hence, $\partial N_{k}(\epsilon)\cap \mathscr{L}\subset \cup_{B_{j}\in S} \overline{B_{j}}\cap \partial N_{k}(\epsilon)$. \end{proof}
 
 \subsection{The finiteness of the set $S$}
 The set $\partial B_{j}\cap \text{Int}~ N_{k}$ equals some $\Sigma^{t}_{i}\subset L_{t}$ for some $t \in \Lambda$.
Let  $S_{t}:=\{B_{j}\in S|~ \partial B_{j}\cap\text{Int}~ N_{k}\subset L_{t} \}$. Then, $S=\amalg_{t\in \Lambda} S_{t}$.

In this subsection, we first show that each $S_{t}$ is finite. Then, we argue that $\{S_{t}\}_{t\in\Lambda}$ contains at most finitely many nonempty sets.
These imply the finiteness of $S$. 
 
\begin{lemma}\label{finite1} Each $S_{t}$ is finite.\end{lemma}
\begin{proof}

We argue by contradiction. Suppose that $S_{t}$ is infinite for some $t$. 

For each $B_{j}\in S_{t}$, there exists  a $i\in I_{t}$ so that $B_{j}$ is equal to $B^{t}_{i}$, where $B^{t}_{i}$ is a component of $N_{k}\setminus \overline{\Sigma^{t}_{i}}$ and $\Sigma^{t}_{i}$ is one component of $L_{t}\cap \text{Int}~N_{k}$. By Proposition \ref{intersect}, $\Sigma^{t}_{i}\cap N_{k}(\epsilon/2)$ is nonempty. 
 
Choose $x_{j}\in\Sigma^{t}_{i}\cap N_{k}(\epsilon/2)$ and $r_{0}=\frac{1}{2}\min\{\epsilon/2, i_{0}\}$, where $i_{0}:=\inf_{x\in N_{k}} \text{Inj}_{M}(x)$.
Then the geodesic ball $B(x_{j}, r_{0})\subset M$ is contained in $N_{k}$.

We apply [Lemma 1, Page 445] in \cite{YM} to the minimal surface $(\Sigma^{t}_{i},\partial \Sigma^{t}_{i})\subset (N_{k}, \partial N_{k})$.  One has that, 
$$\text{Area}( \Sigma^{t}_{i}\cap B(x_{j}, r_{0}))\geq C(r_{0}, i_{0}, K)$$
where $K=\sup_{x\in N_{k}}|K_{M}|$.  One has 
\begin{equation*}
\begin{split}
\int_{L_{t}} \kappa(x)dv
       & \geq\sum_{B_{j}\in S_{t}} \int_{\Sigma^{t}_{i}} \kappa(x) dv\geq \sum_{B_{j}\in S_{t}} \int_{\Sigma^{t}_{i}\cap B(x_{j}, r_{0})} \kappa(x) dv\\
      &\geq \inf_{x\in N_{k}}(\kappa(x))\sum_{B_{j}\in S_{t}} \text{Area}(B(x_{j}, r_{0})\cap \Sigma^{t}_{i})\\
      &\geq C\inf_{x\in N_{k}}(\kappa(x)) |S_{t}|
\end{split}
\end{equation*}
From Theorem \ref{proper}, LHS$\leq 2\pi$, while RHS is infinite. This is a contradiction and finishes the proof.
\end{proof}

\begin{lemma}\label{finite2} $\{S_{t}\}_{t\in \Lambda}$ contains at most finitely many nonempty sets.
\end{lemma}

\begin{proof}We argue by contradiction. Suppose that there exists a sequence $\{S_{t_{n}}\}_{n\in\mathbb{N}}$ of nonempty sets.
For an element $B_{j_{t_{n}}}\in S_{t_{n}}$, there is some $i_{n}\in I_{t_{n}}$ so that $B_{j_{t_n}}$ equals $B^{t_{n}}_{i_{n}}$ where $B^{t_{n}}_{i_{n}}$ is one component of $N_{k}\setminus \overline{\Sigma^{t_{n}}_{i_{n}}}$ and $\Sigma^{t_{n}}_{i_{n}}$ is one of components of $L_{t_{n}}\cap \text{Int}~N_{k}$.

By Proposition \ref{intersect}, $\Sigma^{t_{n}}_{i_{n}}\cap N_{k}(\epsilon/2)$ is not empty. Pick a point $p_{t_{n}}$ in $\Sigma^{t_{n}}_{i_{n}}\cap N_{k}(\epsilon/2)$.

\vspace{2mm}
\noindent\emph{\textbf{Step1}: $\{L_{t_{n}}\}$ subconverges to a lamination $\mathscr{L}'\subset \mathscr{L}$ with finite multiplicity.}\par
\vspace{2mm}

Since $L_{t_{n}}$ is a stable minimal surface, we use  [Theorem 3, Page122] of \cite{Sc} to have that for any compact set $K\subset M$, there is a constant $C_{1}:=C_{1}(K, M, g)$ such that   $$|A_{L_{t_{n}}}|^{2}\leq C_{1}~ \text{on}~ K\cap L_{t_{n}}.$$ 
 Theorem \ref{proper} gives $\int_{L_{t_{n}}}\kappa(x)dv\leq 2\pi$. Hence,  $$\text{Area}(K\cap L_{t_{n}})\leq 2\pi(\inf\limits_{x\in K}\kappa(x))^{-1}.$$

From Theorem \ref{convergence}, $\{L_{t_{n}}\}$ subconverges to a properly embedded lamination $\mathscr{L}'$ with finite multiplicity.  Since $\mathscr{L}$ is a closed set in $M$, $\mathscr{L}'\subset \mathscr{L}$ is a sublamination.\par

From now on, we abuse notation and write $\{L_{t_{n}}\}$ and $\{p_{t_{n}}\}$ for  the convergent subsequences.

\vspace{2mm}
\noindent\emph{\textbf{Step 2}: $\{\Sigma^{t_{n}}_{i_{n}}\}$ converges  with multiplicity one.}\par
\vspace{2mm}

Let $L_{t_{\infty}}$ be the unique component of $\mathscr{L}'$ passing through $p_{\infty}$, where $p_{\infty}=\lim_{n\rightarrow \infty}p_{t_{n}}$. The limit of $\{\Sigma^{t_{n}}_{i_{n}}\}$ is the component $\Sigma_{\infty}$ of $L_{t_{\infty}}\cap N_{k}$ passing through $p_{\infty}$, where $\Sigma^{t_{n}}_{i_{n}}$ is the unique component of $N_{k}\cap L_{t_{n}}$ passing though $p_{t_{n}}$.

\begin{figure}[H]
\begin{center}
\begin{tikzpicture} 
\draw (-9,-4) -- (3, -4);
\draw (-7,4.1) -- (5,4.1);
\draw (-9, -4)--(-7, 4.1);
\draw (3, -4)--(5, 4.1);

\fill[pink]  (-0.9, 0.75) arc(0:360: 1.6 and 1.2);

\draw (2.1, 1) arc (360: 0: 4 and 3);
\draw[dashed] (2,0) arc (0: 169: 4.05 and 3);
\draw (2, 0) arc (365 :160: 4.1 and 3);

\draw (-0.9, 0.75) arc(0:360: 1.6 and 1.2);
\draw[dashed] (-1, -.025) arc(0:360: 1.5 and 1.2);
\node(Os) at (2,-3) {$L_{t_{\infty}}$};
\node(Os) at (-1, -2.5) {$\Sigma_{\infty}$};
\node(Os) at (0, 0) {$\Sigma^{t_{n}}_{i_{n}}$};
\node(Os) at (-2.4, -0.75) {$B^{\Sigma_{\infty}}(p_{\infty})$};
\node(Os) at (-6, -3) {$\pi^{-1}(B^{\Sigma_{\infty}}(p_{\infty}))\cap \Sigma^{t_{n}}_{i_{n}}$};

\draw[->] (-5.8, -2.7) -- (-2.4, 0.5);
 \end{tikzpicture}
 
 \caption{}
\end{center}
\end{figure}
\par

Let $D\subset L_{t_{\infty}}$ be a disc containing ${\Sigma}_{{\infty}}$. (Its existence is ensured by the fact that $L_{t_{\infty}}$ is homeomorphic to $\mathbb{R}^{2}$.) From Definition \ref{converge}, there exists  $\epsilon_{1}>0$ and an integer $N$ such that $$\Sigma^{t_{n}}_{i_{n}}\subset D(\epsilon_{1}),~\text{for}~n\geq N, 
 $$
where $D(\epsilon_{1})$ is the tubular neighborhood of $D$ with radius $\epsilon_{1}$.\par
Let $\pi:D(\epsilon_{1})\rightarrow D$ be the projection. 
By Remark \ref{graph}, for $n$ large enough, the restriction of $\pi$ to each component of $L_{t_{n}}\cap D(\epsilon_{1})$ is injective.

Hence, $\pi|_{\Sigma^{t_{n}}_{i_{n}}}:\Sigma^{t_{n}}_{i_{n}}\rightarrow D$ is injective. That is to say, $\Sigma^{t_{n}}_{i_{n}}$ is a normal graph over a subset of $D$.  
 Therefore, $\{\Sigma^{t_{n}}_{i_{n}}\}$ converges to ${\Sigma}_{\infty}$ with multiplicity one. 
From Definition \ref{converge}, there is a geodesic disc $B^{\Sigma_{\infty}}(p_{\infty})\subset \Sigma_{\infty}$ centered at $p_{\infty}$ so that for $n$ large enough

\vspace{2mm}
\noindent\emph{$(\ast)$: $\pi^{-1}(B^{\Sigma_{\infty}}(p_{\infty}))\cap \Sigma^{t_{n}}_{i_{n}}$ is connected and a normal graph over $B^{\Sigma_{\infty}}(p_{\infty})$. }
\vspace{2mm}

\vspace{2mm}
\noindent\emph{\textbf{Step 3}: Get a contradiction.}\par
\vspace{2mm}
From Definition \ref{lam}, there exists a neighborhood $U$ of $p_{\infty}$ and a coordinate map $\Phi$, such that each component of $\Phi(\mathscr{L}\cap U)$ is $\mathbb{R}^{2}\times \{x\}\cap \Phi(U)$ for some $x\in \mathbb{R}$. Choose the disc $B^{\Sigma_{\infty}}(p_{\infty})$ and $\epsilon_{1}$ small enough such that $\pi^{-1}(B^{\Sigma_{\infty}}(p_{\infty}))\subset U$. We may assume that $U=\pi^{-1}(B^{\Sigma_{\infty}}(p_{\infty}))$.

 From ($\ast$),  $\Sigma^{t_{n}}_{i_{n}}\cap U\subset L_{t_{n}}$ is connected and a graph over $B^{\Sigma_{\infty}}(p_{\infty})$, for $n$ large enough. Since $\partial B_{j_{t_{n}}}\cap U\subset L_{{t_{n}}}$ equals $\Sigma^{t_{n}}_{i_{n}}\cap U$, it is also connected. Therefore $\Phi(\partial B_{j_{t_{n}}}\cap U)$ is the set $\mathbb{R}^{2}\times \{x_{t_{n}}\}\cap \Phi(U)$ for some $x_{t_{n}}\in \mathbb{R}$. In addition, $\Phi(\Sigma_{\infty}\cap U)$ equals $\mathbb{R}^{2}\times\{x_{\infty}\}\cap\Phi(U)$ for some $x_{\infty}\in \mathbb{R}$. Since $\lim\limits_{n\rightarrow \infty}p_{t_{n}}=p_{\infty}$, we have $\lim\limits_{n\rightarrow\infty}x_{t_{n}}=x_{\infty}$.  \par
\begin{figure}[H]
\begin{center}
\begin{tikzpicture} 

\node[draw,ellipse,minimum height=100pt, minimum width = 100pt,thick] (S) at (-4,0){}; 
\node[draw,ellipse,minimum height=100pt, minimum width = 100pt,thick] (S) at (3,0){}; 

\begin{scope}
    \clip (3,0) circle (1.75cm);
    \fill[pink] (1, 1) rectangle (7, -3);
  \end{scope}
  \begin{scope}
  \clip (-4,0) circle (1.75cm);
  \fill[pink] (-8,0.5) rectangle (1, 10);
  \end{scope}
\draw  (1,0) -- (5,0);
\draw (1,0.5) -- (5,0.5);
\draw(1, 1) -- (5,1);

\draw (-6,0) --(-2,0);
\draw (-6, 1)--(-2, 1);
\draw (-6, 0.5)--(-2, 0.5);

\node(Os) at (0.9,0) {$x_{\infty}$};
\node(Os) at (0.9,0.5) {$x_{t_{n'}}$};
\node(Os) at (0.9,1) {$x_{t_{n}}$};

\node(Os) at (-1.7, 0) {$x_{\infty}$};
\node(Os) at (-1.7,1) {$x_{t_{n'}}$};
\node(Os) at (-1.7,0.5) {$x_{t_{n}}$};

\node(OS) at (3, 1) {$\Phi(U)$};
\node(Os) at  (3, -1){$\Phi(U\cap B_{j_{t_{n}}})$};

\node(OS) at (-4, -1) {$\Phi(U)$};
\node(Os) at  (-4, 1){$\Phi(U\cap B_{j_{t_{n}}})$};
\end{tikzpicture}
\caption{}
\end{center}
\end{figure}

The set $U\setminus \partial B_{j_{t_{n}}}$ has two components. Therefore, $\Phi(B_{j_{t_{n}}}\cap U)$ is $\Phi(U)\cap \{x| x_{3}> x_{t_{n}}\}$ or $\Phi(U)\cap \{x| x_{3}< x_{t_{n}}\}$. For $n$ large enough, there exists some $n'\neq n$ such that $\mathbb{R}^{2}\times \{x_{t_{n'}}\}\cap \Phi(U) \subset \Phi(B_{j_{t_{n}}}\cap U)$. This implies that $B_{j_{t_{n}}}\cap B_{j_{t_{n'}}}$ is non-empty.\par
Since $S$ consists of maximal elements in $(\{B^{t}_{i}\}, \subset)$, the set $B_{j_{t_{n}}}\cap B_{j_{t_{n'}}}$ is empty  which leads to a contradiction. This finishes the proof.
\end{proof}

\subsection{The finiteness of $S$ implies Theorem \ref{D}}

We will explain how to deduce Theorem \ref{D} from the finiteness of $S$.\par

 \begin{proof}Let $k_0$ and  $S$ be  assumed as in Section 6.2. Lemma \ref{finite1} and Lemma \ref{finite2} imply that $S$ is finite for $k\geq k_0$. 
 
 We may assume that $\partial B_{j}$ intersects $\partial N_{k}(\epsilon)$ transversally for each $B_{j}\in S$. Remark that each $B_{j}\in S$ equals some $B^{t}_{i}$ and $\partial B_{j}\cap \partial N_{k}(\epsilon)$ is equal to $\Sigma^{t}_{i}\cap \partial N_{k}(\epsilon)$. Since each $\Sigma^{t}_{i}$ is properly embedded (see Theorem \ref{proper}), $\partial N_{k}(\epsilon)\cap(\cup_{B_{j}\in S}\partial B_{j}):=\{c_{i}\}_{i\in I}$ has finitely many components.  Each component is an embedded circle.

 By Lemma \ref{trivial} and Remark \ref{trivial1}, each $c_{i}$ is contractible in $\partial N_{k}(\epsilon)$ and bounds a unique closed disc $D_{i}\subset \partial N_{k}(\epsilon)$ for $k\geq k_0$. The set $(\{D_{i}\}, \subset)$ is a partially ordered set and finite. 
 
Let $\{D_{j'}\}_{j'\in J'}$ be the set of maximal elements, where  $J'$ is finite.
The boundary of $\partial N_{k}(\epsilon)\cap \overline{B}_{j}$ is contained in $\amalg_{j'\in J'}D_{j'}$ since it is  a subset of $\partial B_{j}\cap \partial N_{k}(\epsilon)\subset\amalg_{i\in I}c_{i}$ for each $B_{j}\in S$.

\vspace{1mm}

We have that for any $B_{j}\in S$, $\partial N_{k}(\epsilon)\cap \overline{B_{j}}$ is contained in $\amalg_{j'\in J'}D_{j'}$. The reason is as follows:

If not, $\partial N_{k}(\epsilon)\setminus \amalg_{j'\in J'} D_{j'}$ is contained in $\partial N_{k}(\epsilon)\cap \overline{B}_{j}$ for some $B_{j}\in S$. This implies that the composition of the two maps $\pi_{1}(\partial N_{k}(\epsilon)\setminus (\amalg_{j'\in J'}D_{j'}))\rightarrow \pi_{1}(\overline{B}_{j})\rightarrow\pi_{1}({N}_{k})$ is surjective and not a zero map. This contradicts the fact  the induced map $\pi_{1}({B}_{j})\rightarrow \pi_{1}({N}_{k})$ is trivial.

\vspace{1mm}

Therefore, $\cup_{B_{j}\in S}\overline{B}_{j}\cap \partial N_{k}(\epsilon)$ is contained in $\amalg_{j'\in J'} D_{j'}$. From Proposition \ref{cover}, $\mathscr{L}\cap \partial N_{k}(\epsilon)$ is contained in a disjoint union of finitely many discs $\{D_{j'}\}_{j'\in J'}$. \end{proof}

\section{The proof of The Main Theorem}
In the section, we use Theorem \ref{D} to complete the proof of Theorem \ref{B}. 

Let $(M, g)$ be a complete contractible genus one $3$-manifold, where $M$ is an increasing union of solid tori $\{N_k\}_k$ and the geometric index $I(N_k, N_{k+1})\geq 2$ (see Theorem \ref{genus one}). 

As in Section 5, a meridian  $\gamma_k$ of $N_k$ bounds an embedded disc $\Omega_k\subset N_k$ with the property that $\Omega_k\cap N_{k-1}$ is stable minimal for $g$. Theorem  \ref{limits} tells that $\{\Omega_k\}$ sub-converges towards a minimal lamination $\mathscr{L}$.  Each leaf of $\mathscr{L}$ is a complete  minimal surface. 

\vspace{2mm}

In the following, we explain the proof of Theorem \ref{B}.

\begin{proof} Suppose that a contractible genus one $3$-manifold $M$ has a complete metric $g$ with positive scalar curvature. 

Let $\{N_k\}$, $\{\gamma_k\}$, $\{\Omega_k\}$ and $\mathscr{L}$ be assumed as above. The positivity of scalar curvature shows that each leaf of $L_t$ is a complete stable minimal plane (see Theorem \ref{stability}, Theorem \ref{non-compact} and [Theorem 2, Page 211] in \cite{SY}). In addition, Theorem \ref{D} gives an integer $k_0$ so that  there is a constant $\epsilon>0$ satisfying that $$\mathscr{L}\cap \partial N_{k_0}(\epsilon)\subset\amalg_{i=1}^{m} D_{i},$$ where each $D_{i}$ is a closed disc in $\partial N_{k_0}(\epsilon)$ and $N_{k_0}(\epsilon)$ is defined in Theorem \ref{D}. 

\vspace{2mm}

We now work on the solid torus $N_{k_0}(\epsilon)$.  Let $\theta\subset  \partial N_{k_0}(\epsilon)\setminus \amalg^{m}_{i} D_{i}$ be a longitude  of $N_{k_0}(\epsilon)$. It is also contained in $\partial N_{k_0}(\epsilon)\setminus \mathscr{L}$. 

Since $\mathscr{L}$ is a closed set in $M$, the sets $\mathscr{L}\cap {N}_{k_0+1}$ and $\theta$ are two disjoint closed sets. Choose a neighborhood $U$ of $\mathscr{L}\cap {N}_{k_0+1}$ so that $U\cap \theta=\emptyset$. Since $\{\Omega_{k}\}$ subconverges to $\mathscr{L}$, there exists an integer $k$, large enough, such that $\Omega_{k}\cap \text{Int}~N_{k_0+1}$ is in $U$. In particular, $\Omega_{k}\cap\theta=\emptyset$.  

The set $\Omega_{k}$ is a meridian disc of $N_{k}$ (see Section 5.1). 
The geometric index  $I(N_{k_0}(\epsilon), N_{k})>0$ (for $k$ large enough) implies that there is a meridian $\alpha_k\subset \Omega_{k}\cap \partial N_{k_0}(\epsilon)$ of $N_{k_0}(\epsilon)$ (see Lemma \ref{one}).  The intersection number of $\theta$ and $\alpha_k$ is $\pm 1$. Namely, $\theta$ must intersect $\alpha_{k}$. This contradicts the last paragraph and finishes the proof of Theorem \ref{B}.
\end{proof}

 We use the standard argument in \cite{Kazdan} to have that 
 
 \vspace{2mm}

\noindent\textbf{Corollary~~~~~\ref{C}~~~~~}\emph{:No contractible genus one $3$-manifold admits a complete metric of nonnegative scalar curvature.  }

\section*{Appendix A}

\vspace{1mm} 
\noindent \textbf{Lemma \ref{existence}}   \emph{Let $(\Omega, \partial \Omega) \subset (N, \partial N)$ be a 2-sided embedded disc with some open sub-discs removed, where $N$ is a closed solid torus. Assume that each circle $\gamma_{i}$ is contractible in $\partial N$, where $\partial \Omega=\amalg_{i}\gamma_{i}$. Then $N\setminus \Omega$ has two connected components. Moreover, there is a unique component $B$ satisfying that the induced map $\pi_{1}(B)\rightarrow\pi_{1}(N)$ is trivial. }\par
\vspace{1mm}
\begin{proof} We argue by contradiction. Suppose that $N\setminus \Omega$ is path-connected. That is to say, there is an embedded circle $\sigma\subset N$ which intersects $\Omega$ transversally at one point.\par 
Since each $\gamma_{i}$ is contractible in $\partial N$, it bounds a unique disc $D_{i}\subset\partial N$. The surface $\hat{\Omega}:=\Omega\bigcup_{i}\cup_{\gamma_{i}}D_{i}$ also intersects $\sigma$ transversally at one point. The intersection number between $\hat{\Omega}$ and $\sigma$ is $\pm 1$.\par
However, $\hat{\Omega}$ is the image of a continuous map $g:\mathbb{S}^{2}\rightarrow N$. It is contractible in $N$ (since $\pi_{2}(N)=\{1\}$). The intersection number of $\hat{\Omega}$ and $\sigma$ must be zero, which leads to a contradiction. \par

Therefore, $N\setminus \Omega$ is not connected. Since $\Omega$ is 2-sided and connected, $N\setminus \Omega$ just has two components $B_{0}$ and $B_{1}$.

Next, suppose that the induced map $i_{k}:\pi_{1}(B_{k})\rightarrow \pi_{1}(N)$ is non-trivial for $k=0,1$. 

The image $\text{Im}(i_{k})$ of the map $i_{k}$ is a non-trivial subgroup in $\pi_{1}(N)$. Since $\pi_{1}(N)$ is isomorphic to $\mathbb{Z}$, then $\text{Im}(i_{0})\cap \text{Im}(i_{1})$ is also a non-trivial subgroup. That is to say, there exist two closed curves $c_{0}\subset B_{0}$  and $c_{1}\subset B_{1}$ such that (1) $[c_{0}]$ is non-trivial in $\pi_{1}(N)$; (2) $c_{0}$ is homotopic to $c_{1}$ in $N$.

We can apply the classical theorem of Morrey \cite{M}. It asserts that if $[c_{0}]\neq 0$ in $\pi_{1}(N)$, there is an area-minimizing annulus $A\subset N$ joining $c_{0}$ to $c_{1}$. It is an immersed surface. 

We may suppose that $A$ intersects $\Omega$ transversally and $A\cap \Omega$ is a union of some circles $\{\hat{\gamma}_{j}\}_{j}$. The boundary of $A_{0}:=A\cap B_{0}$ is a disjoint union of $c_{0}$ and $\cup_j\hat{\gamma}_{j}$. Therefore, $[c_{0}]=\sum_{j}[\hat{\gamma}_{j}]$ in $H_{1}(N)\cong \pi_{1}(N)$. 

We consider the image $\hat{\Omega}$ of a $2$-sphere as mentioned above. The map $\pi_{1}(\Omega)\rightarrow \pi_{1}(\hat{\Omega})$ is trivial. Therefore, the map $\pi_{1}(\Omega)\rightarrow \pi_{1}(\hat{\Omega})\rightarrow \pi_{1}(N)$ is also trivial. 

Therefore, each $\hat{\gamma}_{j}$ is nullhomotopic in $N$. Thus, $[c_{0}]=1$ in $\pi_{1}(N)$, which contradicts to the choice of $c_{0}$. 
We conclude that one of the induced maps $\pi_{1}(B_{i})\rightarrow\pi_{1}(N)$ is trivial. \par
Finally, we prove the uniqueness. Suppose that the two induced maps are both trivial. This implies that for each $i$, the map $H_{1}(B_{i})\rightarrow H_{1}(N)$ is trivial. 
 Applying Mayer-Vietoris Sequence to $N=B_{0}\cup_{\Omega}B_{1}$, one has: 
 $$H_{1}(B_{0})\oplus H_{1}(B_{1})\rightarrow H_{1}(N)\rightarrow \hat{H}_{0}(\Omega).$$
 Since $\Omega$ is connected, $\hat{H}_{0}(\Omega)$ is trivial. Therefore, $H_{1}(N)$ is also trivial.  
  This contradicts the fact that $H_{1}(N)$ is isomorphic to $\mathbb{Z}$.  This completes the proof.
\end{proof}

Consider two disjoint surfaces $(\Omega_{1}, \partial\Omega_{1})$ and $(\Omega_{2}, \partial\Omega_{2})$ as assumed in Lemma \ref{existence}. For $t=1,2$, $N \setminus \Omega_{t}$ has  two components. Let $B_{t}$ be the unique component of $N\setminus \Omega_{t}$ satisfying $\pi_{1}(B_{t})\rightarrow \pi_{1}(N)$ is trivial. One has:

\vspace{1mm}
\noindent\textbf{Lemma \ref{sep}}  \emph{Let $(\Omega_{1}, \partial\Omega_{1})$ and $(\Omega_{2}, \partial\Omega_{2})$ be two disjoint surfaces as  in Lemma \ref{existence}. Assume that for  $t=1,2$, $N\setminus \Omega_{t}$ has a unique component $B_{t}$ with the property that the map $\pi_{1}(B_{t})\rightarrow \pi_{1}(N)$ is trivial. Then one of the following holds:
\begin{enumerate}
\item $B_{1}\cap B_{2}=\emptyset$;
\item $B_{1}\subset B_{2}$;
\item$B_{2}\subset B_{1}$.
\end{enumerate} 
}
\vspace{1mm}
\begin{proof} Suppose $B_{1}\cap B_{2}$ and $B_{1}\setminus B_{2}$ are both nonempty. Say, there are two points $p_{1}\in B_{1}\setminus  B_{2}$ and $p_{2}\in B_{1}\cap B_{2}$. 

First, $\Omega_{2}$ is contained in $B_{1}$. 
The reason follows as below: $B_{1}$ includes a curve $\gamma$ joining $p_{1}$ and $p_{2}$ (since $B_{1}$ is connected). The curve $\gamma$ must intersect $\Omega_{2}$ at some point(s). 
Hence, $\Omega_{2}\cap B_{1}$ is not empty. 
Since $\Omega_{1}\cap \Omega_{2}$ is empty, $\Omega_{2}$ lies in one of component of $N\setminus \Omega_{1}$. 
Therefore, $\Omega_{2}$ is contained in $B_{1}$.\par
Second, $\Omega_{2}$ cuts $B_{1}$ into two components. Otherwise, there is a circle in $B_{1}$ which intersects $\Omega_{2}$ at one point. As argued in Lemma \ref{existence}, such a circle does not exist. \par
Finally, take the component $B$ of $B_{1}\setminus \Omega_{2}$ satisfying that $\partial B\cap \Omega_{1}$ is empty. Then, $B$ is also a component of $N\setminus \Omega_{2}$. In addition, the map $\pi_{1}(B)\rightarrow \pi_{1}(B_{1})\rightarrow \pi_{1}(N)$ is trivial. From the uniqueness of $B_{2}$, one has that $B=B_{2}$. This implies $B_{2}\subset B_{1}$. 
\end{proof}

\section*{Appendix B} In this appendix, we will prove Theorem \ref{stability} and Theorem \ref{non-compact}. Recall that $(M, g)$ is a complete contractible genus one $3$-manifold, where $M$ is an increasing union of solid tori $\{N_{k}\}_{k}$. 

For each $k$, there is a meridian $\gamma_{k}$ of $N_k$ and a metric $g_{k}$ on $N_k$ so that 
(1) $(N_k, g_k)$ has mean convex boundary;
(2) $g_{k}|_{N_{k-1}}=g|_{N_{k-1}}$. We use  Theorem \ref{stable} to find an area-minimizing disc $\Omega_k$ in $(N_k, g_k)$ with boundary $\gamma_k$.  The sequence $\{\Omega_k\}$ sub-converges to the minimal lamination $\mathscr{L}$. Each leaf is a complete minimal surface. (See Section 5)

For our convenience, we may assume that the sequence $\{\Omega_k\}_k$ converges to $\mathscr{L}$.

\vspace{1mm}

\noindent\textbf{Theorem 5.8.}
\emph{Each leaf in $\mathscr{L}$ is stable minimal. }

\vspace{1mm}

As laminations, $\Omega_k$ converges to $\mathscr{L}$ in $C^{1, \alpha}$-topology. However, $\Omega_k$ locally converges to minimal hypersurfaces in $C^{\infty}$-topology. Since each $\Omega_k$ is stable minimal, then we have that each leaf in $\mathscr{L}$ is stable minimal. 

\vspace{2mm}

\noindent\textbf{Theorem 5.9.} \emph{If $(M, g)$ has positive scalar curvature, each leaf in $\mathscr{L}$ is non-compact.}

\begin{proof} We argue by contradiction. Suppose that there exists a compact leaf $L_{t}$ in $\mathscr{L}$.

\vspace{2mm}

\textbf{Step 1}: Topology of $L_{t}$

\vspace{2mm}

 The positivity of the scalar curvature shows that $L_{t}$ is a $2$-sphere or a projective plane (see [Theorem 5.1, Page 139] in \cite{SY3}).

If $L_{t}$ is a projective plane,  $L_{t}$ is 1-sided. Hence, $M\setminus L_{t}$ is connected. There is a an embedded curve $\gamma$ in $M$ which intersects $L_{t}$ transversally at one point. The intersection numberof $L_{t}$ and $\gamma$ is $\pm 1$. 

However, $\gamma$ is homotopically trivial, since $M$ is contractible. Thus,  the intersection number of $\gamma$ and $L_t$ is zero, a contradiction.

We can conclude that $L_t$ is a $2$-sphere. 

\vspace{2mm}

\textbf{Step 2:} Area Estimate. 

\vspace{2mm}

Since  $M$ is contractible, it  is irreducible (see Section 2.1). Then $L_t$ bounds a $3$-ball $B$. Let $N_{2\epsilon}(B)$ be the tubular neighborhood of $B$ with radius $2\epsilon$. The set $N_{2\epsilon}(L_t)$ is a subset of $N_{2\epsilon}(B)$. 

Since $N_{2\epsilon}(B)$ is relatively compact, there is a positive integer $k_{0}$ so that $N_{2\epsilon}(B)\subset N_{k_{0}-1}$.

\vspace{2mm}

From now on, we fix the integer $k> k_{0}$. Let $\{\Sigma_k^j\}_j$ be the set of components of $\Omega_k\cap N_{2\epsilon}(L_{t})$. 

\vspace{2mm}

In the following, we show that there is a constant $C$, independent of $k$ and $j$, so that the area of $\Sigma_{k}^{j}$ in $(M, g)$ is less than $C$.

\vspace{2mm}

We may assume that $\Omega_k$ intersects $\partial N_{2\epsilon}(B)$ transversally. The intersection $\Omega_k\cap \partial N_{2\epsilon}(B):=\{c_i\}_i$ has finitely many components. Each component $c_{i}$ is a circle and bounds a  closed disc $D_{i}$ in $\Omega_{k}$.

\vspace{2mm}

Since $\partial N_{2\epsilon}(B)$ is a $2$-sphere, there is an embedded disc $D'_{i}\subset \partial N_{2\epsilon}(B)$ with boundary $c_i$. 
We have that  for any $D_{i}$,
\begin{equation*}
\text{Area}(D_{i}, g_{k})\leq \text{Area}(D'_{i}, g_{k}),
\end{equation*}where $\text{Area}(D_{i}, g_{k})$ is the area of $D_i$ in $(N_k, g_k)$.
 (If not, we consider the disc $(\Omega_{k}\setminus D_{i})\cup_{\gamma_i}D'_i$ with boundary $\partial \Omega_{k}$. Its area is less than the area of $\Omega_{k}$ in $(N_{k}, g_k)$, a contradiction. )

\vspace{2mm}

Let $C^{max}$ be the set of maximal circles of $\{c_i\}_i$ in $\Omega_{k}$. We have that 
\begin{equation*}
\Omega_{k}\cap N_{2\epsilon}(L_t)\subset \Omega_{k}\cap N_{2\epsilon}(B)\subset \amalg_{\gamma_i\in C^{max}}D_{i}.\end{equation*}
Then, $\Sigma_{k}^j$ is a subset of some $D_{i}$. 

\vspace{2mm}

Recall that $g_{k}|_{N_{2\epsilon}(B)}=g|_{N_{2\epsilon}(B)}$ for $k>k_0$. For each $k$ and $j$, we have that 
\begin{equation*}
\begin{split}
&\text{Area}(\Sigma_{k}^{j}, g_{k})=\text{Area}(\Sigma_k^j, g)\\
&\text{Area}(\partial N_{2\epsilon}(B), g_{k})=\text{Area}(\partial N_{2\epsilon}(B), g)
\end{split}\end{equation*}

We then have that 
\begin{equation*}
\begin{split}
\text{Area}(\Sigma_{k}^j, g)&=\text{Area}(\Sigma_{k}^j, g_k)\leq \text{Area}(D_{i}, g_{k})\\
& \leq \text{Area}(D'_i, g_{k})\leq \text{Area}(\partial N_{2\epsilon}(B), g_{k})\\
&=\text{Area}(\partial N_{2\epsilon}(B), g). 
\end{split}
\end{equation*}

We conclude that for each $k> k_0$ and $j$, 
\begin{equation*}
\text{Area}(\Sigma_{k}^j, g)\leq \text{Area}(\partial N_{2\epsilon}(B), g). 
\end{equation*} 

\vspace{2mm}

\textbf{Step 3}: Contradiction. 

\vspace{2mm}

Choose a point $p\in L_t$ and a point $p_{k}\in \Omega_{k}\cap N_{\epsilon}(L_t)$ so that $\lim_{k\rightarrow \infty}p_k=p$. 

Let $\Sigma_{k}^{j_{k}}$ be the component of $\Omega_k\cap N_{2\epsilon}(L_t)$ passing through $p_{k}$. As the proof in Step 2, we have that for $k> k_0$,
\begin{equation*}
\text{Area}(\Sigma_{k}^{j_{k}}, g)\leq \text{Area}(\partial N_{2\epsilon}(B), g). 
\end{equation*}

From the result of Schoen \cite{Sc}, the curvatures of these surfaces $\{\Sigma_k^{j_{k}}\}_k$ are uniformly bounded in $N_{2\epsilon}(L_t)$. 
By Theorem \ref{convergence}, the sequence $\{\Sigma_{k}^{j_{k}}\}$ sub-converges smoothly to a properly embedded surface $\Sigma$ with finite multiplicity in $N_{\epsilon}(L_t)$. 

For our convenience, we assume that $\{\Sigma_{k}^{j_k}\}$ converges smoothly to  $\Sigma$  in $N_{\epsilon}(L_t)$.  The limit $\Sigma\subset \mathscr{L}$ is a disjoint union of embedded surfaces. Its boundary is contained in $\partial N_{\epsilon}(L_t)$. In addition, $p$ is contained in $\Sigma$. Hence, $L_t$ is one component  of $\Sigma$. 

Since $\Sigma$ is properly embedded, the set $\Sigma':=\Sigma\setminus L_t$ is a closed set.  The sets $\Sigma'$ and $L_{t}$ are two disjoint closed sets. 
Choose $\delta<\epsilon/2$ small enough so that 
\begin{equation*}
N_{2\delta}(L_{t})\cap N_{2\delta}(\Sigma')=\emptyset.
\end{equation*}

\vspace{2mm}

\noindent\textbf{Claim:} For $k$ large enough, $\Sigma_{k}^{j_{k}}$ is contained in $N_{\delta}(L_{t})\amalg (N_{2\epsilon}(L_{t})\setminus N_{2\delta}(L_t))$. 

\vspace{2mm}

Since $\Sigma_{k}^{j_k}$ is a subset of $N_{2\epsilon}(L_t)$, $\Sigma_{k}^{j_{k}}\setminus N_{2\delta}(L_t)$ is contained in $N_{2\epsilon}(L_{t})\setminus N_{2\delta}(L_t)$.  It is sufficient to show that $\Sigma_{k}^{j_{k}}\cap N_{2\delta}(L_t)$ is contained in $N_{\delta}(L_t)$. 

\vspace{2mm}

For $k$ large enough, $\Sigma_{k}^{j_k}\cap N_{\epsilon}(L_t)$ is contained in $ N_{\delta}(\Sigma)$, since $\Sigma^{j_k}_k$ converges to $\Sigma$ in $N_{\epsilon}(L_t)$. Hence, $\Sigma_k^{j_k}\cap N_{2\delta}(L_t)$ is a subset of $N_{2\delta}(L_{t})\cap N_{\delta}(\Sigma)$.
From the choice of $\delta$, we have that 

a) $N_{\delta}(\Sigma)$ is equal to $N_{\delta}(L_t)\amalg N_{\delta}(\Sigma')$;

b) $N_{\delta}(\Sigma')\cap N_{2\delta}(L_{t})$ is empty.

\vspace{2mm}

By a), $N_{2\delta}(L_t)\cap N_{\delta}(\Sigma)$ is equal to $N_{\delta}(L_t)\amalg (N_{\delta}(\Sigma')\cap N_{2\delta}(L_{t}))$. From b), it is equal to $N_{\delta}(L_t)$.  
Therefore, $\Sigma^{k}_{j_{k}}\cap N_{2\delta}(L_t)$ is contained in $N_{\delta}(L_t)$.

This finishes the proof of the claim. 

\vspace{2mm}

For $k$ large enough, $p_{k}$ is contained in $N_{\delta}(L_{t})$. Namely, $\Sigma_{k}^{j_k}\cap N_{\delta}(L_t)$ is non-empty. In addition, since $\partial \Sigma_{k}^{j_k}\subset \partial N_{2\epsilon}(L_t)$ is non-empty, $\Sigma_{k}^{j_k}\cap (N_{2\epsilon}(L_t)\setminus N_{2\delta}(L_t))$ is also nonempty. 

The sets $N_{\delta}(L_t)$ and $\overline{N_{2\epsilon}(L_t)\setminus N_{2\delta}(L_t)}$ are disjoint. Since $\Sigma_{k}^{j_k}$ is connected, we use the claim to have that one of these two sets, $\Sigma_{k}^{j_k}\cap N_{\delta}(L_t)$ and $\Sigma_{k}^{j_k}\cap (N_{2\epsilon}(L_t)\setminus N_{2\delta}(L_t))$, is empty. This is in contradiction with the last paragraph.  

We can conclude that each leaf $L_t$ is non-compact. 
\end{proof}

\section*{Acknowledgement}

I would like to express my deep gratitude to my supervisor, G\'erard Besson for suggesting this problem, also for his endless support, constant encouragement. I also thank Harold Rosenberg, Laurent Mazet and  Yunhui Wu for helpful discussions related to this work.  

I  am also very grateful to the referees for many helpful comments that improved the exposition of this paper. This research is supported by ERC Advanced Grant 320939, GETOM and partially by  the Special Priority Program SPP 2026 `` Geometry at infinity ".


\begin{thebibliography}{{Coh}99}
\bibitem[And85]{And} Michael Anderson,
{\it Curvature estimates for minimal surfaces in $3 $-manifolds},
Annales scientifiques de l'{\'E}cole Normale Sup{\`e}rieure. {\bf 18} (1985) 89—105,
MR0803196, Zbl 0578.49027.

\bibitem[BBM11]{BBM}Laurent Bessi\`eres and G\'erard Besson and Sylvain Maillot, 
{\it Ricci flow on open 3-manifolds and positive scalar curvature}, 
Geom. Topol. {\bf 15} (2011) 927—975,
MR2821567, Zbl 1237.53064.

\bibitem[CM04]{CM} Tobias  Colding and William  Minicozzi
{\it The space of embedded minimal surfaces of fixed genus in a 3-manifold. IV: Locally simply connected}, 
Ann. Math. (2) {\bf 160} (2004) 573–615, 
 MR2123933, Zbl 1076.53069. 

\bibitem[CM11]{CM1} Tobias Colding and William Minicozzi, 
{\it A course in minimal surfaces}, 
American Mathematical Soc.  { \bf 121}, (2011), 
MR2780140,  Zbl 1242.53007. 


\bibitem[{Coh}35]{CV} Stefan {Cohn-Vossen}, 
{{\it K\"urzeste Wege und Totalkr\"ummung auf Fl\"achen.}},
{Compos. Math.}, {\bf 2}, (1935) 69--133,
MR1556908,  Zbl 0011.22501. 



\bibitem[FCS80]{FS} D. Fischer-Colbrie  and Richard Schoen,
{\it The structure of complete stable minimal surfaces in 3-manifolds of non-negative scalar curvature}, 
Communications on Pure and Applied Mathematics, {\bf 33}, (1980) 199–-211,
MR0562550,  Zbl 0439.53060.


\bibitem[{Fis}85]{F} D. {Fischer-Colbrie},
{\it On complete minimal surfaces with finite Morse index in three manifolds},	
Invent. Math., {\bf 82}, (1985) 121--132,
MR808112,  Zbl 0573.53038.

\bibitem[Fre96]{FKR} Frensel, Katia Rosenvald,
{\it Stable complete surfaces with constant mean curvature}, 
Boletim da Sociedade Brasileira de Matem{a}tica, {\bf 27}, (1996) 129-144,
MR1418929, Zbl 0890.53054.

\bibitem[GL83]{GL} Mikhael Gromov and Blaine Lawson,
{\it Positive scalar curvature and the Dirac operator on complete Riemannian manifolds},
Publications Math{\'e}matiques de l'Institut des Hautes {\'E}tudes Scientifiques, {\bf 58}, (1983) ,83–-196,
MR0720933, Zbl 0538.53047. 


\bibitem[GRW18]{GRW} Garity, Dennis and Repov{\v{s}}, Du{\v{s}}an and Wright, David,
{\it Contractible 3-manifolds and the double 3-space property},	
Transactions of the American Mathematical Society, {\bf 370} (2018), 2039—2055
MR3739201, Zbl 1381.57015. 

\bibitem[GT15]{GT}
David Gilbarg and Neil Trudinger.
{\it Elliptic partial differential equations of second order}
 Springer, (2015), 
 MR1814364, Zbl 1436.00023.

\bibitem[Hat00]{HA}
Allen Hatcher.
{\it Notes on basic 3-manifold topology}, 2000.

\bibitem[HP71]{HP}
LS~Husch and TM~Price.
{\it Addendum to: Finding a boundary for a 3-manifold},
Annals of Mathematics, {\bf 93} 486--488 (1971),
MR0286113,  Zbl 0214.22301.

\bibitem[Kaz82]{Kazdan}
Jerry Kazdan.
 {\it Deformation to positive scalar curvature on complete manifolds},
Mathematische Annalen, {\bf 261(2)} 227--234 (1982),
MR0675736,  Zbl 0476.53022.

\bibitem[McM62]{Mc}
DR~McMillan.
{\it Some contractible open 3-manifolds},
Transactions of the American Mathematical Society,
 {\bf 102(2)}, 373--382 (1962), 
 MR0137105, Zbl 0111.18701.

\bibitem[Mor09]{M}
Charles Morrey.
{\it Multiple integrals in the calculus of variations}.
Springer Science \& Business Media (2008),
MR2492985, Zbl 1213.49002.

\bibitem[MRR02]{MRR}
William Meeks, Antonio Ros, and Harold Rosenberg.
 {\it The global theory of minimal surfaces in flat spaces}.
Springer, (2002)
MR1901611,  Zbl 0983.00044.

\bibitem[MY80]{YM}
William Meeks and Shing-Tung Yau.
{\it Topology of three dimensional manifolds and the embedding problems in
  minimal surface theory},
 Annals of Mathematics, {\bf 112(3)} 441--484(1980)
 MR0595203,  Zbl 0458.57007.
 
\bibitem[MY82]{YM1}
William Meeks and Shing-Tung Yau.
{\it The existence of embedded minimal surfaces and the problem of
  uniqueness},
 Mathematische Zeitschrift, {\bf 179(2)}151--168 (1982),
 MR0645492,  Zbl 0479.49026.

\bibitem[Per02a]{P1}
Grisha Perelman.
{\it The entropy formula for Ricci flow and its geometric applications},
 arXiv preprint math/0211159 (2002),
 Zbl 1130.53001.

\bibitem[Per02b]{P2}
Grisha Perelman.
{\it Ricci flow with surgery on three-manifolds},
 arXiv preprint math/0303109 (2002),
 Zbl 1130.53002.

\bibitem[Per03]{P3}
Grisha Perelman. 
{\it Finite extinction time for the solutions to the Ricci flow on certain
  three-manifolds},
arXiv preprint math/0307245, (2003),
 Zbl 1130.53003.

\bibitem[Rol03]{Rol}
Dale Rolfsen.
 {\it Knots and links}, 
American Mathematical Soc. {\bf 346}, (1990)
MR1277811, Zbl 0854.57002.

\bibitem[Rot12]{R}
Joseph Rotman.
 {\it An introduction to the theory of groups}, 
Springer-Verlag (1995),
MR1307623,  Zbl 0810.20001.

\bibitem[Sch53]{Sch}
Horst Schubert.
{\it Knoten und vollringe}
 Acta Mathematica, {\bf 90(1)},131--286 (1953)
MR0072482,   Zbl 0051.40403.

\bibitem[Sch83]{Sc}
Richard Schoen.
{\it Estimates for stable minimal surfaces in three dimensional manifolds}
 In {\em Seminar on minimal submanifolds}, {\bf103}  111--126.
  Princeton University Press Princeton, NJ, (1983)
 MR0795231,  Zbl 0532.53042 .

\bibitem[Sta72]{S}
John Stallings.
{\it Group theory and three-dimensional manifolds}.
 Yale University Press, (1972)
MR0415622, Zbl 0241.57001.

\bibitem[SY79a]{SY3}
Richard Schoen and Shing-Tung Yau.
{\it Existence of incompressible minimal surfaces and the topology of
  three dimensional manifolds with non-negative scalar curvature}
 Annals of Mathematics, {\bf110(1)}, 127--142, (1979),
 MR0541332, Zbl 0431.53051.

\bibitem[SY79b]{SY1}
Richard Schoen and Shing-Tung Yau.
{\it On the structure of manifolds with positive scalar curvature}
 Manuscripta mathematica, {\bf28} 159--183,(1979),
 MR0535700, Zbl 0423.53032.

\bibitem[SY82]{SY}
Richard Schoen and Shing~Tung Yau.
{\it Complete three-dimensional manifolds with positive Ricci curvature
  and scalar curvature}
 In {\em Seminar on Differential Geometry}, {\bf102}, 
  209--228. Princeton Univ. Press Princeton, NJ, (1982)
 MR0645740,  Zbl 0481.53036.

\bibitem[Whi35]{Wh}
J.H.C. Whitehead.
{\it A certain open manifold whose group is unity}, 
 { The Quarterly Journal of Mathematics}, {\bf6(1)}, 268--279(1935).





\end{thebibliography}
\end{document}